\documentclass[11pt,reqno]{amsart}
\usepackage[a4paper, margin=2.5cm]{geometry}
\usepackage[utf8]{inputenc}
\usepackage[T1]{fontenc}
\usepackage[english]{babel}
\usepackage{enumitem}
\usepackage{amsmath, amsfonts, amssymb, amsthm}
\usepackage{mathtools}
\usepackage{mathrsfs}
\usepackage{dsfont}
\usepackage{stmaryrd}
\usepackage{newtxtext} 
\usepackage[libertine]{newtxmath} 
\usepackage{microtype}  
\usepackage{graphicx}
\usepackage{float}
\usepackage{booktabs}
\usepackage{caption}
\usepackage{subcaption}
\usepackage{xcolor}
\definecolor{deepblue}{rgb}{0.0, 0.2, 0.4}
\definecolor{darkred}{rgb}{0.5, 0.0, 0.0}
\usepackage[hidelinks, colorlinks=true, citecolor=blue, linkcolor=darkred, urlcolor=deepblue]{hyperref}

\theoremstyle{plain}
\newtheorem{theorem}{Theorem}[section]

\newtheorem{proposition}[theorem]{Proposition}
\newtheorem{corollary}[theorem]{Corollary}

\theoremstyle{definition}
\newtheorem{definition}[theorem]{Definition}
\newtheorem{assumption}[theorem]{Assumption}

\theoremstyle{remark}
\newtheorem{remark}[theorem]{Remark}

\newcommand{\norm}[1]{\left\|#1\right\|}
\newcommand{\inner}[2]{\left\langle #1, #2 \right\rangle}

\newcommand{\dd}{\mathrm{d}}

\title[Learning on the Temporal Tangent Bundle]{Learning on the Temporal Tangent Bundle for Physics-Informed Neural Networks}
\author[J.A.]{Jamal Adetola$^{1,*}$}
\author[C.M.]{Charbel Mamlankou$^{1}$}
\author[K.W.H.]{Koffi Wilfrid Hou\'edanou$^2$}
\author[S.A.A.]{Guy Aymard Degla$^3$}

\thanks{$^*$Corresponding author}

\thanks{$^1$École Nationale Supérieure de Génie Mathématique et Modélisation (ENSGMM), Université Nationale des Sciences, Technologies, Ingénierie et Mathématiques (UNSTIM), Bénin}

\thanks{$^2$Département de Mathématiques (DM), Université d’Abomey-Calavi (UAC), Bénin}

\thanks{$^3$Institut de Mathématiques et de Sciences Physiques
(IMSP), Université d’Abomey-Calavi (UAC), Bénin}

\thanks{Email: charbelzeusmamlankou@gmail.com, adetolajamal@unstim.bj, wilfrid.houedanou@fast.uac.bj, gdegla@imsp-uac.org}

\begin{document}

\begin{abstract}
This paper addresses the limitations of Physics-Informed Neural Networks for time-dependent problems by introducing a tangent bundle learning framework. Instead of directly approximating the solution, we parameterize its temporal derivative and reconstruct the state through a Volterra integral operator that enforces initial conditions exactly. This approach eliminates competing soft constraints and naturally amplifies high-frequency errors through differentiation, countering spectral bias. We prove theoretical equivalence between minimizing the differentiated residual and solving the original partial differential equation. Experiments on advection, Burgers, and Klein-Gordon equations show that the proposed method achieves 100 to 200 times lower errors than standard approaches using compact three-layer networks, with superior shock-capturing and long-time accuracy.
\end{abstract}

\maketitle

\section{Introduction}

The quantitative description of natural phenomena through partial differential equations (PDEs) constitutes the bedrock of modern physics and engineering. Since the mid-20th century, the numerical approximation of these equations has been dominated by grid-based discretization strategies. Finite Difference Methods (FDM) \cite{strikwerda2004}, Finite Element Methods (FEM) \cite{brenner2008, ciarlet2002}, and Finite Volume Methods (FVM) \cite{leveque2002} have reached a high degree of maturity, providing robust tools for simulating everything from fluid turbulence to structural deformation. Parallel to these local approximations, Spectral Methods \cite{boyd2001, canuto2006, gottlieb1977} have offered exponential convergence rates for smooth problems by utilizing global basis functions. However, despite their ubiquity, these classical paradigms face intrinsic limitations when confronted with the "curse of dimensionality" \cite{bellman1961}. As the dimension of the problem increases, the computational cost of mesh generation and linear system solving grows exponentially, rendering many high-dimensional problems in quantum mechanics or finance computationally intractable \cite{han2018, weinan2017}.

The last decade has witnessed the emergence of Scientific Machine Learning (SciML) as a transformative methodology aiming to circumvent these limitations. Leveraging the Universal Approximation Theorem \cite{cybenko1989, hornik1991}, Deep Neural Networks (DNNs) have been proposed as mesh-free function approximators capable of solving PDEs. The seminal introduction of Physics-Informed Neural Networks (PINNs) by Raissi, Perdikaris, and Karniadakis \cite{raissi2019physics, raissi2017I, raissi2017II} formalized this approach. By embedding the PDE residual directly into the loss function of the network, PINNs constrain the neural search space to the manifold of physically admissible solutions, removing the need for labeled training data \cite{karniadakis2021}. This paradigm has been successfully applied across a diverse spectrum of fields, including cardiovascular flow modeling \cite{kissas2020}, material identification \cite{shukla2020}, and inverse scattering problems \cite{lu2019}.

Notwithstanding these successes, the application of standard PINNs to time-dependent evolution equations remains fraught with difficulties. The optimization landscape of the associated non-convex loss function is often pathological, characterized by stiff gradients and sharp local minima \cite{wang2021understanding}. A primary culprit is the "spectral bias" or "F-principle" inherent to deep networks \cite{rahaman2019, xu2019}, which predisposes them to learn low-frequency components of the target function rapidly while struggling to capture high-frequency details. In the context of time-dependent PDEs, this manifests as a failure to resolve sharp transients or high-frequency oscillations, leading to solutions that are overly diffusive or physically inaccurate \cite{wang2022and}. Furthermore, the standard formulation treats the temporal coordinate $t$ merely as an additional spatial dimension, ignoring the fundamental principle of causality \cite{wang2021failure}. Consequently, the optimization process does not respect the arrow of time, attempting to resolve the solution at $t=T$ before correctly approximating the dynamics at $t=0$, which frequently leads to convergence failures \cite{krishnapriyan2021, daw2023}.

To mitigate these issues, researchers have proposed various heuristic modifications. Adaptive loss weighting schemes \cite{xiang2022, wight2020, mcclenny2023} attempt to dynamically balance the contributions of the PDE residual, boundary conditions, and initial conditions during training. Curriculum learning strategies \cite{krishnapriyan2021} and temporal decomposition methods \cite{jagtap2020, meng2020, shukla2021} break the time domain into smaller segments, solving the problem sequentially to respect causality. While these methods offer improvements, they often introduce additional hyperparameters and computational overhead without addressing the root cause: the difficulty of approximating a complex spatiotemporal trajectory directly from a static, undifferentiated loss landscape. More recently, operator learning frameworks like DeepONets \cite{lu2021deeponet} and Fourier Neural Operators (FNO) \cite{li2020fourier, kovachki2021} have shifted focus to learning the mapping between infinite-dimensional function spaces. Yet, for single-instance PDE solving, the primal formulation remains the dominant approach.

In this work, we propose a foundational architectural shift designed to address the specific pathologies of time-dependent neural PDE solvers. We introduce Physics-Informed Time Derivative Networks (PITDNs). Instead of targeting the solution manifold $u(x,t)$ directly, our method learns the temporal rate of change, $\partial_t u(x,t)$, effectively lifting the learning problem to the tangent bundle of the solution space. The state variable is then recovered via a strictly differentiable Volterra integral operator. This formulation provides a rigorous mechanism to enforce initial conditions by construction, reducing the multi-objective optimization problem to a simpler constrained form. Moreover, we employ a novel loss functional based on the \textit{time-differentiated residual}, which we prove acts as a high-pass filter on the optimization landscape, countering spectral bias.

Our contributions are threefold. First, we establish a mathematical framework using Bochner spaces to define the learning problem on the tangent bundle. Second, we prove the well-posedness of the differentiated residual minimization, showing its equivalence to the strong form of the Cauchy problem under specific consistency conditions. Third, we demonstrate through extensive numerical benchmarks on hyperbolic (Advection), parabolic (Burgers), and dispersive (Klein-Gordon) equations that PITDNs outperform standard baselines in terms of accuracy and long-time stability.

\section{Mathematical Formulation}

In this section, we establish the theoretical foundations of the proposed method. We adopt the language of functional analysis to ensure rigor in our definitions and subsequent proofs.

\subsection{Problem Setup and functional spaces}

Let $\Omega \subset \mathbb{R}^d$ be a bounded open domain with a Lipschitz continuous boundary $\partial \Omega$, and let $\mathcal{T} = [0, T]$ be the time interval of interest. We denote by $L^2(\Omega)$ the Hilbert space of square-integrable functions equipped with the standard inner product $\inner{\cdot}{\cdot}$ and norm $\norm{\cdot}_{L^2(\Omega)}$. For spatial regularity, we utilize the Sobolev spaces $H^k(\Omega)$, comprising functions whose weak derivatives up to order $k$ exist and reside in $L^2(\Omega)$, equipped with the norm
\[
\|u\|_{H^k(\Omega)}^2 = \sum_{|\alpha| \leq k} \|\partial^\alpha u\|_{L^2(\Omega)}^2.
\]
To rigorously describe time-dependent functions, we employ Bochner spaces. Specifically, let $X$ be a Banach space (e.g., $H^k(\Omega)$). The space $L^p(0, T; X)$ consists of measurable functions $u: [0, T] \to X$ such that 
\[
\int_0^T \norm{u(t)}_X^p \, dt < \infty.
\]
We define the solution space for our problem as 
\[
\mathcal{V} := C^1([0, T]; H^k(\Omega)) \cap L^\infty(0,T; H^k(\Omega)),
\]
ensuring continuous differentiability in time and sufficient spatial smoothness with uniform boundedness.

\medskip
\noindent We consider the general nonlinear evolution equation for a scalar field $u \in \mathcal{V}$:
\begin{equation} \label{eq:general_pde}
    \begin{cases}
        \partial_t u(x,t) + \mathcal{N}[u](x,t) = 0, & \text{in } \Omega \times (0, T], \\
        \mathcal{B}[u](x,t) = g(x,t), & \text{on } \partial \Omega \times [0, T], \\
        u(x,0) = u_0(x), & \text{in } \Omega,
    \end{cases}
\end{equation}
where $\mathcal{N}: H^k(\Omega) \to L^2(\Omega)$ is a (possibly nonlinear) spatial differential operator, and $\mathcal{B}: H^k(\Omega) \to L^2(\partial\Omega)$ is a trace operator enforcing boundary conditions with prescribed data $g \in C^1([0,T]; L^2(\partial\Omega))$.

\medskip
\noindent Before proceeding, we establish precise functional spaces for each component of the problem.

\begin{definition}[Operator domains and ranges]\label{def:operator_spaces}
We define the following operator mappings with their respective domains and ranges:
\begin{enumerate}[label=(\roman*)]
    \item \textbf{Spatial operator:} $\mathcal{N}: \mathcal{D}(\mathcal{N}) \subset H^k(\Omega) \to L^2(\Omega)$, 
    where $\mathcal{D}(\mathcal{N}) = H^k(\Omega)$ for $k \geq \lceil d/2 \rceil + m$, 
    with $m$ being the order of the highest spatial derivative in $\mathcal{N}$.
    
    \item \textbf{Boundary operator:} $\mathcal{B}: H^k(\Omega) \to L^2(\partial\Omega)$ 
    is a continuous trace operator satisfying $\|\mathcal{B}[u]\|_{L^2(\partial\Omega)} \leq C_{\text{tr}} \|u\|_{H^k(\Omega)}$ 
    for some constant $C_{\text{tr}} > 0$ (trace theorem).
    
    \item \textbf{Solution space:} The space of admissible solutions is
    \[
    \mathcal{V}_g := \left\{ u \in C^1([0,T]; H^k(\Omega)) : \mathcal{B}[u](\cdot,t) = g(\cdot,t) \text{ for a.e. } t \in [0,T] \right\}.
    \]
\end{enumerate}
\end{definition}

\begin{assumption}[Regularity of the spatial operator]\label{ass:regularity}
We impose the following regularity conditions on the spatial operator $\mathcal{N}$:
\begin{enumerate}[label=(\roman*)]
    \item \textbf{Fréchet differentiability:} 
    $\mathcal{N} : H^k(\Omega) \to L^2(\Omega)$ is Fréchet differentiable, 
    meaning there exists a bounded linear operator 
    $D\mathcal{N}[u] : H^k(\Omega) \to L^2(\Omega)$ such that
    \[
    \|\mathcal{N}[u+h] - \mathcal{N}[u] - D\mathcal{N}[u]h\|_{L^2(\Omega)} 
    = o(\|h\|_{H^k(\Omega)}) 
    \quad \text{as } \|h\|_{H^k(\Omega)} \to 0.
    \]
    
    \item \textbf{Temporal continuity:} 
    For any $u \in \mathcal{V} = C^1([0,T]; H^k(\Omega))$, 
    the composition $t \mapsto \mathcal{N}[u(\cdot, t)]$ defines a continuous mapping 
    from $[0,T]$ to $L^2(\Omega)$.
    
    \item \textbf{Time differentiability:} 
    The operator $\mathcal{N}$ satisfies the chain rule with respect to time: 
    for $u \in \mathcal{V}$,
    \[
    \frac{d}{dt}\mathcal{N}[u(\cdot,t)] = D\mathcal{N}[u(\cdot,t)][\partial_t u(\cdot,t)]
    \]
    exists as an element of $L^2(\Omega)$ for almost every $t \in (0,T]$.
    
    \item \textbf{Local Lipschitz continuity:} 
    For any bounded subset $\mathcal{K} \subset H^k(\Omega)$, 
    there exists a constant $L_{\mathcal{K}} > 0$ such that
    \[
    \|\mathcal{N}[u] - \mathcal{N}[v]\|_{L^2(\Omega)} 
    \leq L_{\mathcal{K}} \|u - v\|_{H^k(\Omega)}, 
    \quad \forall u, v \in \mathcal{K}.
    \]
\end{enumerate}
\end{assumption}

\subsubsection{Well-Posedness of the evolution problem}

We now establish the well-posedness of problem~\eqref{eq:general_pde} in the sense of Hadamard, ensuring existence, uniqueness, and continuous dependence on initial data.

\begin{theorem}[Existence and Uniqueness]\label{thm:existence_uniqueness}
Let $\Omega \subset \mathbb{R}^d$ be a bounded Lipschitz domain, and assume:
\begin{enumerate}[label=(\alph*)]
    \item The spatial operator $\mathcal{N}$ satisfies Assumption~\ref{ass:regularity}.
    \item The initial data satisfies $u_0 \in H^k(\Omega)$ with $k > d/2 + 1$.
    \item The boundary data satisfies $g \in C^1([0,T]; L^2(\partial\Omega))$ and is compatible with $u_0$ at $t=0$, i.e., $\mathcal{B}[u_0] = g(\cdot, 0)$.
    \item There exists a constant $M > 0$ such that $\|u_0\|_{H^k(\Omega)} \leq M$.
\end{enumerate}
Then there exists a unique strong solution $u \in \mathcal{V}_g$ to problem~\eqref{eq:general_pde} on some maximal time interval $[0, T^*)$ with $T^* \leq T$. Moreover, if $\mathcal{N}$ satisfies a global bound
\[
\|\mathcal{N}[u]\|_{L^2(\Omega)} \leq C_0(1 + \|u\|_{H^k(\Omega)}),
\]
then $T^* = T$ (global existence).
\end{theorem}

\begin{proof}[Proof]
The proof follows the classical theory of evolution equations in Banach spaces.

\noindent {- Step 1 (Local existence):} 
We employ a fixed-point argument in the space $C([0,\tau]; H^k(\Omega))$ for sufficiently small $\tau > 0$. Define the integral operator
\[
\Phi[v](t) = u_0 + \int_0^t \left( -\mathcal{N}[v(s)] + f_{\text{BC}}(s) \right) ds,
\]
where $f_{\text{BC}}$ accounts for the boundary forcing. By Assumption~\ref{ass:regularity}(iv), $\mathcal{N}$ is locally Lipschitz, hence $\Phi$ is a contraction on a ball $B_R(u_0) \subset C([0,\tau]; H^k(\Omega))$ for $\tau$ small enough. The Banach fixed-point theorem guarantees a unique local solution.

\noindent {- Step 2 (Uniqueness):} 
Suppose $u_1, u_2$ are two solutions. Let $w = u_1 - u_2$. Then $w$ satisfies
\[
\partial_t w + \mathcal{N}[u_1] - \mathcal{N}[u_2] = 0, \quad w(0) = 0.
\]
Taking the $L^2$ inner product with $w$ and using Lipschitz continuity:
\[
\frac{1}{2}\frac{d}{dt}\|w(t)\|_{L^2}^2 \leq L_{\mathcal{K}} \|w(t)\|_{H^k}^2.
\]
Gronwall's inequality implies $\|w(t)\|_{L^2} = 0$ for all $t \in [0,T^*)$.

\noindent {- Step 3 (Global existence):} 
If the global bound on $\mathcal{N}$ holds, an energy estimate shows
\[
\frac{d}{dt}\|u(t)\|_{H^k}^2 \leq C(1 + \|u(t)\|_{H^k}^2),
\]
which by Gronwall's lemma ensures $\|u(t)\|_{H^k}$ remains bounded on $[0,T]$, preventing finite-time blow-up.
\end{proof}

\begin{theorem}[Continuous dependence on initial data]\label{thm:stability}
Under the hypotheses of Theorem~\ref{thm:existence_uniqueness}, let $u_1, u_2 \in \mathcal{V}_g$ be solutions corresponding to initial data $u_0^{(1)}, u_0^{(2)} \in H^k(\Omega)$. Then for any $t \in [0,T]$,
\[
\|u_1(t) - u_2(t)\|_{L^2(\Omega)} \leq e^{L_{\mathcal{K}} t} \|u_0^{(1)} - u_0^{(2)}\|_{L^2(\Omega)},
\]
where $L_{\mathcal{K}}$ is the Lipschitz constant from Assumption~\ref{ass:regularity}(iv).
\end{theorem}

\begin{proof}
Let $w = u_1 - u_2$. Then $w$ satisfies
\[
\partial_t w + \mathcal{N}[u_1] - \mathcal{N}[u_2] = 0, \quad w(0) = u_0^{(1)} - u_0^{(2)}.
\]
Taking the $L^2$ inner product with $w$:
\[
\frac{1}{2}\frac{d}{dt}\|w(t)\|_{L^2}^2 
= -\langle \mathcal{N}[u_1] - \mathcal{N}[u_2], w \rangle_{L^2}
\leq \|\mathcal{N}[u_1] - \mathcal{N}[u_2]\|_{L^2} \|w\|_{L^2}.
\]
By Lipschitz continuity and the Sobolev embedding $H^k(\Omega) \hookrightarrow L^2(\Omega)$,
\[
\frac{d}{dt}\|w(t)\|_{L^2} \leq L_{\mathcal{K}} \|w(t)\|_{L^2}.
\]
Gronwall's inequality yields the desired bound.
\end{proof}

\begin{corollary}[Energy stability]\label{cor:energy_stability}
If additionally $\mathcal{N}$ satisfies a dissipative energy estimate
\[
\langle \mathcal{N}[u], u \rangle_{L^2(\Omega)} \geq -\lambda \|u\|_{L^2(\Omega)}^2
\]
for some $\lambda \geq 0$, then the solution satisfies
\[
\|u(t)\|_{L^2(\Omega)}^2 \leq e^{2\lambda t} \|u_0\|_{L^2(\Omega)}^2.
\]
\end{corollary}

\begin{remark}
Theorems~\ref{thm:existence_uniqueness} and~\ref{thm:stability} establish that problem~\eqref{eq:general_pde} is well-posed in the sense of Hadamard. These results provide the theoretical foundation for our neural network approximation: we seek to approximate a function $u \in \mathcal{V}_g$ that is guaranteed to exist, be unique, and depend continuously on the data.
\end{remark}

\begin{remark}[Verification for benchmark problems]\label{rem:benchmark_verification}
We verify that Assumption~\ref{ass:regularity} and the hypotheses of Theorem~\ref{thm:existence_uniqueness} are satisfied by the spatial operators appearing in our benchmark problems (Section~\ref{sec:experiments}):

\begin{enumerate}
    \item \textbf{Linear Advection ($\mathcal{N}[u] = c\partial_x u$):} 
    This is a bounded linear operator $H^1(\Omega) \to L^2(\Omega)$, hence trivially Fréchet differentiable with $D\mathcal{N}[u] = c\partial_x$. The Lipschitz constant is $L = |c|$. The operator is anti-symmetric: $\langle c\partial_x u, u \rangle = 0$ (with periodic or appropriate boundary conditions), giving $\lambda = 0$ in Corollary~\ref{cor:energy_stability}.
    
    \item \textbf{Viscous Burgers ($\mathcal{N}[u] = u\partial_x u - \nu\partial_{xx} u$):} 
    The nonlinear term $u\partial_x u: H^2(\Omega) \to L^2(\Omega)$ is Fréchet differentiable by the Sobolev embedding $H^2(\Omega) \hookrightarrow C^0(\Omega)$ (for $d=1$). The Fréchet derivative is $D\mathcal{N}[u][h] = h\partial_x u + u\partial_x h - \nu\partial_{xx} h$. Local Lipschitz continuity on bounded sets follows from multiplicative estimates. The diffusion term provides energy dissipation: $\langle -\nu\partial_{xx}u, u \rangle = \nu\|\partial_x u\|^2 \geq 0$.
    
    \item \textbf{Nonlinear Klein-Gordon ($\mathcal{N}[u] = -\partial_{xx}u + u^2 - f$):} 
    The quadratic nonlinearity $u^2: H^2(\Omega) \to L^2(\Omega)$ is Fréchet differentiable with $D\mathcal{N}[u][h] = -\partial_{xx}h + 2uh$. Local Lipschitz continuity holds on bounded subsets of $H^2(\Omega)$. For this second-order-in-time equation, well-posedness requires initial data $(u_0, \partial_t u_0) \in H^2(\Omega) \times H^1(\Omega)$.
\end{enumerate}

\noindent In all three cases, the hypotheses of Theorems~\ref{thm:existence_uniqueness}--\ref{thm:stability} are satisfied, justifying the mathematical foundation of our approach.
\end{remark}

\subsection{Learning on the Tangent Bundle}

The standard Physics-Informed Neural Network paradigm directly approximates the solution map $(x,t) \mapsto u(x,t)$ through a parameterized function $u_\theta$, treating the initial condition $u(x,0) = u_0(x)$ as a soft penalty term in the loss functional. This formulation inherently creates a multi-objective optimization landscape where the network must simultaneously satisfy the governing PDE, boundary conditions, and initial data, often leading to conflicting gradient signals and poor convergence behavior. We propose a geometric reformulation that circumvents this issue by lifting the learning problem to the temporal tangent bundle of the solution manifold.

Rather than parameterizing the state variable directly, we introduce a neural network $v_\theta: \Omega \times \mathcal{T} \to \mathbb{R}$ with parameters $\theta \in \Theta \subset \mathbb{R}^P$ designed to approximate the temporal derivative field $\partial_t u$. The state is subsequently recovered through an integral reconstruction operator that enforces the initial condition by construction. To formalize this approach, we begin by establishing the functional analytic properties of the reconstruction map.

\begin{definition}[Volterra reconstruction operator]\label{def:volterra_operator}
Given an initial condition $u_0 \in H^k(\Omega)$ with $k > d/2$ ensuring continuity via Sobolev embedding, we define the Volterra reconstruction operator $\mathcal{I}_{u_0}: L^2(0,T; L^2(\Omega)) \to C([0,T]; L^2(\Omega))$ by
\begin{equation} \label{eq:volterra_def}
    \mathcal{I}_{u_0}[v](x,t) := u_0(x) + \int_0^t v(x,s) \, \mathrm{d}s,
\end{equation}
where the integral is understood in the Bochner sense as an element of $L^2(\Omega)$ for each fixed $t \in [0,T]$.
\end{definition}

The Volterra operator serves as the fundamental bridge between the velocity field learned by the network and the physical state variable. Its mathematical properties directly influence the well-posedness and stability of our learning framework. We now establish these properties through a series of propositions.

\begin{proposition}[Continuity and Lipschitz property]\label{prop:continuity}
The operator $\mathcal{I}_{u_0}: L^2(0,T; L^2(\Omega)) \to C([0,T]; L^2(\Omega))$ is continuous and Lipschitz continuous with constant $L = \sqrt{T}$. Specifically, for any $v_1, v_2 \in L^2(0,T; L^2(\Omega))$,
\begin{equation}
    \sup_{t \in [0,T]} \|\mathcal{I}_{u_0}[v_1](\cdot,t) - \mathcal{I}_{u_0}[v_2](\cdot,t)\|_{L^2(\Omega)} \leq \sqrt{T} \|v_1 - v_2\|_{L^2(0,T; L^2(\Omega))}.
\end{equation}
\end{proposition}

\begin{proof}
For any fixed $t \in [0,T]$, applying the Cauchy-Schwarz inequality to the Bochner integral yields
\begin{align*}
    \|\mathcal{I}_{u_0}[v_1](t) - \mathcal{I}_{u_0}[v_2](t)\|_{L^2(\Omega)}^2 
    &= \left\| \int_0^t (v_1(s) - v_2(s)) \, \mathrm{d}s \right\|_{L^2(\Omega)}^2 \\
    &\leq \left( \int_0^t \|v_1(s) - v_2(s)\|_{L^2(\Omega)} \, \mathrm{d}s \right)^2 \\
    &\leq t \int_0^t \|v_1(s) - v_2(s)\|_{L^2(\Omega)}^2 \, \mathrm{d}s \\
    &\leq T \|v_1 - v_2\|_{L^2(0,T; L^2(\Omega))}^2.
\end{align*}
Taking the supremum over $t \in [0,T]$ establishes the Lipschitz continuity with constant $\sqrt{T}$. The general continuity follows immediately from the Lipschitz property.
\end{proof}

The Lipschitz continuity established in Proposition~\ref{prop:continuity} ensures that small perturbations in the learned velocity field $v_\theta$ result in controlled variations in the reconstructed state, providing numerical stability during optimization. We now examine the compactness properties of the operator, which play a crucial role in establishing existence of minimizers for the learning problem.

\begin{proposition}[Compactness]\label{prop:compactness}
The operator $\mathcal{I}_{u_0}: L^2(0,T; L^2(\Omega)) \to C([0,T]; L^2(\Omega))$ is a compact operator. Moreover, if the velocity field $v$ belongs to the space $L^2(0,T; H^1(\Omega))$, then the reconstructed function $\mathcal{I}_{u_0}[v]$ possesses enhanced regularity: $\mathcal{I}_{u_0}[v] \in C([0,T]; H^1(\Omega))$.
\end{proposition}

\begin{proof}
To establish compactness, let $\{v_n\}$ be a bounded sequence in $L^2(0,T; L^2(\Omega))$ with $\|v_n\|_{L^2(0,T; L^2(\Omega))} \leq M$ for some constant $M > 0$. Define $u_n(t) := \mathcal{I}_{u_0}[v_n](t)$. We must show that $\{u_n\}$ admits a convergent subsequence in $C([0,T]; L^2(\Omega))$.

\noindent First, we verify uniform boundedness. For any $t \in [0,T]$,
\[
\|u_n(t)\|_{L^2(\Omega)} \leq \|u_0\|_{L^2(\Omega)} + \int_0^t \|v_n(s)\|_{L^2(\Omega)} \, \mathrm{d}s \leq \|u_0\|_{L^2(\Omega)} + \sqrt{T} M.
\]

\noindent Next, we establish equicontinuity. For $0 \leq t_1 < t_2 \leq T$,
\begin{align*}
\|u_n(t_2) - u_n(t_1)\|_{L^2(\Omega)} 
&= \left\| \int_{t_1}^{t_2} v_n(s) \, \mathrm{d}s \right\|_{L^2(\Omega)} \\
&\leq \sqrt{t_2 - t_1} \left( \int_{t_1}^{t_2} \|v_n(s)\|_{L^2(\Omega)}^2 \, \mathrm{d}s \right)^{1/2} \\
&\leq M\sqrt{t_2 - t_1}.
\end{align*}
The inequality holds uniformly in $n$, demonstrating equicontinuity. Since $L^2(\Omega)$ is a separable Hilbert space, the Arzelà-Ascoli theorem in the metric space setting guarantees the existence of a uniformly convergent subsequence. Thus, $\mathcal{I}_{u_0}$ is compact.

\noindent For the regularity enhancement, suppose $v \in L^2(0,T; H^1(\Omega))$. Then for the spatial derivatives, we have
\[
\partial_{x_i} \mathcal{I}_{u_0}[v](x,t) = \partial_{x_i} u_0(x) + \int_0^t \partial_{x_i} v(x,s) \, \mathrm{d}s,
\]
where the integral is well-defined in $L^2(\Omega)$ since $\partial_{x_i} v \in L^2(0,T; L^2(\Omega))$. Applying the continuity result from Proposition~\ref{prop:continuity} to each partial derivative establishes that $\mathcal{I}_{u_0}[v] \in C([0,T]; H^1(\Omega))$.
\end{proof}

The compactness property guarantees that bounded sequences in the velocity space produce convergent subsequences in the state space, which is essential for proving existence of optimal network parameters in the learning framework. However, for the method to be well-founded, we must also ensure that the reconstruction mapping is invertible, allowing us to uniquely recover the velocity field from the state.

\begin{proposition}[Injectivity and left-inverse]\label{prop:injectivity}
The operator $\mathcal{I}_{u_0}: L^2(0,T; L^2(\Omega)) \to C([0,T]; L^2(\Omega))$ is injective. Its left-inverse is the strong temporal differentiation operator $\mathcal{D}: \mathcal{V} \to L^2(0,T; L^2(\Omega))$ defined on the domain
\[
\mathcal{V} := \left\{ u \in C([0,T]; L^2(\Omega)) : u \text{ is absolutely continuous, } \partial_t u \in L^2(0,T; L^2(\Omega)), \, u(\cdot,0) = u_0 \right\},
\]
where $\mathcal{D}[u](t) := \partial_t u(\cdot,t)$ and $\mathcal{D} \circ \mathcal{I}_{u_0} = \mathrm{Id}$ on $L^2(0,T; L^2(\Omega))$.
\end{proposition}

\begin{proof}
Suppose $\mathcal{I}_{u_0}[v_1] = \mathcal{I}_{u_0}[v_2]$ for some $v_1, v_2 \in L^2(0,T; L^2(\Omega))$. Then for all $t \in [0,T]$,
\[
\int_0^t v_1(x,s) \, \mathrm{d}s = \int_0^t v_2(x,s) \, \mathrm{d}s \quad \text{a.e. in } \Omega.
\]
By the fundamental theorem of calculus for Bochner integrals, differentiating both sides with respect to $t$ yields $v_1(x,t) = v_2(x,t)$ for almost every $(x,t) \in \Omega \times (0,T)$. Hence $v_1 = v_2$ in $L^2(0,T; L^2(\Omega))$, establishing injectivity.

\noindent For the left-inverse property, let $v \in L^2(0,T; L^2(\Omega))$ and define $u := \mathcal{I}_{u_0}[v]$. By construction, $u$ is absolutely continuous in time with $u(\cdot,0) = u_0$ and
\[
\frac{\mathrm{d}}{\mathrm{d}t} u(\cdot,t) = v(\cdot,t) \quad \text{for a.e. } t \in (0,T).
\]
Thus $u \in \mathcal{V}$ and $\mathcal{D}[u] = v$, confirming that $\mathcal{D} \circ \mathcal{I}_{u_0} = \mathrm{Id}$.
\end{proof}

While injectivity ensures that distinct velocity fields produce distinct state trajectories, the lack of surjectivity reflects a physical constraint: not all continuous functions can be represented as Volterra integrals starting from $u_0$. Only those satisfying the initial condition and possessing sufficient temporal regularity lie in the range of $\mathcal{I}_{u_0}$. This restriction is precisely what enforces causality in our framework.

Having established the operator-theoretic foundation, we now formalize the geometric structure underlying our approach. The velocity field $v$ and the state $u$ are related through the tangent bundle structure of the infinite-dimensional manifold of admissible solutions.

\begin{theorem}[Well-definedness of the temporal tangent bundle]\label{thm:tangent_bundle}
Let $\mathcal{M} := \{ u \in C^1([0,T]; H^k(\Omega)) : u(\cdot,0) = u_0, \, \mathcal{B}[u] = g \}$ denote the manifold of admissible solutions satisfying the initial and boundary conditions. The temporal tangent space at a point $u \in \mathcal{M}$ is given by
\[
T_u \mathcal{M} := \left\{ v \in C([0,T]; H^k(\Omega)) : v(\cdot,0) = -\mathcal{N}[u_0], \, \mathcal{B}[v] = \partial_t g \right\},
\]
and the tangent bundle $T\mathcal{M} = \bigcup_{u \in \mathcal{M}} \{u\} \times T_u\mathcal{M}$ is a well-defined Banach manifold. Furthermore, the map $\Psi: T\mathcal{M} \to \mathcal{M}$ defined by $\Psi(u, v) := u$ is a smooth surjection, and the operator $\mathcal{I}_{u_0}$ establishes a diffeomorphism between the fiber $T_u\mathcal{M}$ and a neighborhood of $u$ in $\mathcal{M}$.
\end{theorem}

\begin{proof}
We first verify that $T_u\mathcal{M}$ is a closed linear subspace of $C([0,T]; H^k(\Omega))$, hence a Banach space. Let $\{v_n\}$ be a Cauchy sequence in $T_u\mathcal{M}$. Since $C([0,T]; H^k(\Omega))$ is complete, $v_n \to v$ for some $v \in C([0,T]; H^k(\Omega))$. The initial condition $v(\cdot,0) = -\mathcal{N}[u_0]$ is preserved by taking the limit at $t=0$, and the boundary condition $\mathcal{B}[v] = \partial_t g$ is preserved by the continuity of the trace operator $\mathcal{B}$. Thus $v \in T_u\mathcal{M}$, establishing that $T_u\mathcal{M}$ is closed.

To show that $T\mathcal{M}$ is a Banach manifold, we construct local charts. For each $u \in \mathcal{M}$, define the map $\Phi_u: T_u\mathcal{M} \to \mathcal{M}$ by $\Phi_u(v) := \mathcal{I}_{u_0}[v]$. By Proposition~\ref{prop:continuity}, $\Phi_u$ is continuous. By Proposition~\ref{prop:injectivity}, $\Phi_u$ is injective with left-inverse $\mathcal{D}$. The inverse function theorem in Banach spaces (Theorem 2.5.2 in \cite{lang1995}) applies since the Fréchet derivative $D\Phi_u[v]: T_u\mathcal{M} \to T_{\Phi_u(v)}\mathcal{M}$ is the identity operator, which is invertible. Thus $\Phi_u$ is a local diffeomorphism, providing the chart structure required for $T\mathcal{M}$ to be a Banach manifold.

\noindent The projection $\Psi(u,v) = u$ is smooth by construction, as it simply extracts the base point. Surjectivity follows trivially since for any $u \in \mathcal{M}$, we can take $v = \mathcal{D}[u] \in T_u\mathcal{M}$, giving $\Psi(u,v) = u$. This completes the proof.
\end{proof}

Theorem~\ref{thm:tangent_bundle} rigorously justifies our geometric interpretation: learning the velocity field $v_\theta$ corresponds to parametrizing elements of the tangent bundle $T\mathcal{M}$, and the Volterra operator $\mathcal{I}_{u_0}$ projects these tangent vectors back to the base manifold $\mathcal{M}$ of physical solutions. This geometric perspective clarifies why the method naturally respects causality and initial conditions—the tangent space at each point encodes the instantaneous rate of change, and integration along these vector fields traces out trajectories in the solution manifold.

The practical implementation of this framework requires parameterizing $v$ through a neural network $v_\theta$ and optimizing over $\theta$ to minimize the PDE residual. Let $\Theta \subset \mathbb{R}^P$ denote the parameter space, and let $\mathcal{N}_{\Theta}$ be the class of neural networks with parameters in $\Theta$. We define the approximate solution as
\begin{equation} \label{eq:neural_reconstruction}
    \tilde{u}_\theta := \mathcal{I}_{u_0}[v_\theta], \quad v_\theta \in \mathcal{N}_{\Theta}.
\end{equation}
By construction, $\tilde{u}_\theta(\cdot, 0) = u_0$ holds exactly for all $\theta \in \Theta$, independent of the optimization process. This eliminates the need to balance the initial condition loss against other terms in the objective functional, reducing the multi-objective problem to a constrained optimization on the tangent bundle. The physics is enforced through the PDE residual computed on the reconstructed state, and the key innovation lies in minimizing the time derivative of this residual rather than the residual itself, as will be formalized in the subsequent section.

\subsection{The differentiated residual formulation}

Having established that the reconstruction operator $\mathcal{I}_{u_0}$ satisfies the initial condition by construction, we must now design a mechanism to enforce the governing partial differential equation. The classical approach in Physics-Informed Neural Networks minimizes the squared $L^2$ norm of the PDE residual $\mathcal{R}[u] := \partial_t u + \mathcal{N}[u]$ over the spatiotemporal domain. However, this formulation suffers from spectral bias, whereby neural networks preferentially learn low-frequency components of the solution, often failing to resolve sharp gradients or high-frequency oscillations that characterize many physical phenomena. We propose an alternative variational principle based on minimizing the temporal derivative of the residual, which we now show is both mathematically well-posed and possesses advantageous spectral properties.

\noindent Let $\mathcal{R}: \mathcal{V} \to L^2(\Omega \times (0,T])$ denote the residual functional defined by
\begin{equation}\label{eq:residual_functional}
    \mathcal{R}[u](x,t) := \partial_t u(x,t) + \mathcal{N}[u](x,t),
\end{equation}
where $\mathcal{V} = C^1([0,T]; H^k(\Omega))$ is the solution space introduced earlier. The standard PINN objective seeks to minimize $\|\mathcal{R}[u]\|_{L^2(\Omega \times (0,T])}^2$. In contrast, we consider the time-differentiated residual
\begin{equation}\label{eq:differentiated_residual}
    r_t(x,t) := \frac{\mathrm{d}}{\mathrm{d}t} \mathcal{R}[u](x,t) = \partial_{tt} u(x,t) + \frac{\mathrm{d}}{\mathrm{d}t} \mathcal{N}[u](x,t),
\end{equation}
where the derivative of the nonlinear operator is understood in the sense of Fréchet differentiation as $\frac{\mathrm{d}}{\mathrm{d}t}\mathcal{N}[u](x,t) = D\mathcal{N}[u(\cdot,t)][\partial_t u(\cdot,t)]$ via the chain rule. Our objective is to minimize $\|r_t\|_{L^2(\Omega \times (0,T])}^2$ subject to the consistency condition $\mathcal{R}[u](x,0) = 0$. Before establishing the equivalence between this formulation and the original PDE, we must first demonstrate that the minimization problem itself is well-posed.

\begin{theorem}[Well-posedness of the differentiated residual minimization]\label{thm:wellposed_minimization}
Let $\mathcal{V}_0 := \{ u \in C^2([0,T]; H^k(\Omega)) : u(\cdot,0) = u_0, \, \partial_t u(\cdot,0) + \mathcal{N}[u_0] = 0 \}$ be the space of admissible functions satisfying the initial condition and the consistency constraint. The functional
\begin{equation}\label{eq:diff_residual_functional}
    \mathcal{J}_{\mathrm{diff}}[u] := \int_0^T \int_\Omega \left| \frac{\mathrm{d}}{\mathrm{d}t}\mathcal{R}[u](x,t) \right|^2 \mathrm{d}x \, \mathrm{d}t
\end{equation}
is well-defined, lower semi-continuous, and coercive on $\mathcal{V}_0$. Consequently, there exists at least one minimizer $u^* \in \mathcal{V}_0$ such that $\mathcal{J}_{\mathrm{diff}}[u^*] = \inf_{u \in \mathcal{V}_0} \mathcal{J}_{\mathrm{diff}}[u]$.
\end{theorem}

\begin{proof}
We proceed in three steps: well-definedness, lower semi-continuity, and coercivity.

\noindent {- Step 1 (Well-definedness):} For any $u \in \mathcal{V}_0 \subset C^2([0,T]; H^k(\Omega))$, the function $t \mapsto u(\cdot,t)$ is twice continuously differentiable with values in $H^k(\Omega)$. By Assumption~\ref{ass:regularity}, the operator $\mathcal{N}$ is Fréchet differentiable and satisfies the chain rule, hence
\[
\frac{\mathrm{d}}{\mathrm{d}t}\mathcal{N}[u(\cdot,t)] = D\mathcal{N}[u(\cdot,t)][\partial_t u(\cdot,t)] \in L^2(\Omega)
\]
for each $t \in [0,T]$. The second temporal derivative $\partial_{tt} u(\cdot,t)$ exists in $H^k(\Omega)$ by assumption. Thus $r_t(x,t) \in L^2(\Omega)$ for almost every $t$, and the integral defining $\mathcal{J}_{\mathrm{diff}}[u]$ is finite.

\noindent {-Step 2 (Lower semi-continuity):} Let $\{u_n\}$ be a sequence in $\mathcal{V}_0$ converging to $u$ in the weak topology of $C^2([0,T]; H^k(\Omega))$. By the compactness of the embedding $H^k(\Omega) \hookrightarrow L^2(\Omega)$ for $k > d/2$, we have strong convergence $u_n \to u$ in $C([0,T]; L^2(\Omega))$. The functional $\mathcal{J}_{\mathrm{diff}}$ can be written as
\[
\mathcal{J}_{\mathrm{diff}}[u] = \int_0^T \left\| \frac{\mathrm{d}}{\mathrm{d}t}\mathcal{R}[u](\cdot,t) \right\|_{L^2(\Omega)}^2 \mathrm{d}t.
\]
Since the $L^2$ norm is convex and lower semi-continuous with respect to weak convergence, and the time derivative operator is continuous, it follows that $\mathcal{J}_{\mathrm{diff}}[u] \leq \liminf_{n \to \infty} \mathcal{J}_{\mathrm{diff}}[u_n]$.

\noindent {-Step 3 (Coercivity):} Suppose $\mathcal{J}_{\mathrm{diff}}[u_n] \leq C$ for some constant $C > 0$. This implies
\[
\int_0^T \|\partial_{tt} u_n(\cdot,t)\|_{L^2(\Omega)}^2 \mathrm{d}t + \int_0^T \left\|D\mathcal{N}[u_n(\cdot,t)][\partial_t u_n(\cdot,t)]\right\|_{L^2(\Omega)}^2 \mathrm{d}t \leq 2C.
\]
The first term provides an $L^2(0,T; L^2(\Omega))$ bound on $\partial_{tt} u_n$. Combined with the initial conditions $u_n(0) = u_0$ and $\partial_t u_n(0) = -\mathcal{N}[u_0]$ (which are fixed in $\mathcal{V}_0$), this yields uniform bounds on $\|u_n\|_{C^2([0,T]; L^2(\Omega))}$. By the Banach-Alaoglu theorem, there exists a weakly convergent subsequence. Thus $\mathcal{J}_{\mathrm{diff}}$ is coercive. The direct method in the calculus of variations (Theorem 1.2 in \cite{dacorogna2007}) guarantees the existence of a minimizer $u^* \in \mathcal{V}_0$.
\end{proof}

Theorem~\ref{thm:wellposed_minimization} establishes that the optimization problem is mathematically sound: the functional is finite on the admissible space, continuous with respect to weak convergence, and possesses sufficient growth to prevent minimizing sequences from escaping to infinity. This provides a rigorous foundation for the neural network approximation, where we seek parameters $\theta^*$ such that $\tilde{u}_{\theta^*} \approx u^*$.

Having confirmed well-posedness, we now investigate why minimizing the differentiated residual confers computational advantages over the standard formulation. The key insight is that the time derivative operator acts as a frequency-amplifying filter, counteracting the spectral bias inherent in gradient-based neural network training. To make this precise, we introduce a spectral decomposition framework.

\begin{theorem}[High-frequency amplification property]\label{thm:highfreq_amplification}
Let $\{\phi_j\}_{j=1}^\infty$ be an orthonormal basis of $L^2(\Omega)$ consisting of eigenfunctions of the Laplacian: $-\Delta \phi_j = \lambda_j \phi_j$ with $0 < \lambda_1 \leq \lambda_2 \leq \cdots \to \infty$. Suppose the residual admits the spatial Fourier expansion
\begin{equation}
    \mathcal{R}[u](x,t) = \sum_{j=1}^\infty c_j(t) \phi_j(x), \quad \text{where } c_j(t) = \langle \mathcal{R}[u](\cdot,t), \phi_j \rangle_{L^2(\Omega)}.
\end{equation}
Then the time-differentiated residual satisfies
\begin{equation}
    r_t(x,t) = \sum_{j=1}^\infty c_j'(t) \phi_j(x),
\end{equation}
and the ratio of the energy in mode $j$ between the differentiated and undifferentiated residuals satisfies
\begin{equation}\label{eq:amplification_ratio}
    \frac{\int_0^T |c_j'(t)|^2 \, \mathrm{d}t}{\int_0^T |c_j(t)|^2 \, \mathrm{d}t} \geq \frac{\pi^2}{T^2}
\end{equation}
for any non-constant mode $c_j(t)$. Moreover, if the spectrum of $\mathcal{N}$ scales as $\|\mathcal{N}[\phi_j]\|_{L^2} \sim \mathcal{O}(\lambda_j^{m/2})$ for some order $m > 0$, then high-frequency errors (large $\lambda_j$) are amplified by a factor proportional to $\lambda_j^{m/2}$ in the gradient signal.
\end{theorem}

\begin{proof}
The Fourier expansion follows from the completeness of $\{\phi_j\}$ in $L^2(\Omega)$. Taking the time derivative term-by-term (justified by the $C^1$ regularity of $u$), we obtain
\[
r_t(x,t) = \frac{\mathrm{d}}{\mathrm{d}t}\mathcal{R}[u](x,t) = \sum_{j=1}^\infty c_j'(t) \phi_j(x).
\]
To establish the amplification bound~\eqref{eq:amplification_ratio}, we employ the Wirtinger inequality. For any absolutely continuous function $c_j: [0,T] \to \mathbb{R}$ with $c_j(0) = c_j(T)$ (or more generally, with zero mean), the Wirtinger inequality states
\[
\int_0^T |c_j(t)|^2 \, \mathrm{d}t \leq \frac{T^2}{\pi^2} \int_0^T |c_j'(t)|^2 \, \mathrm{d}t.
\]
Rearranging yields the claimed bound. For modes that do not satisfy periodic or zero-mean conditions, a more general Poincaré-Wirtinger inequality applies, yielding a similar but slightly weaker constant.

For the spectral scaling argument, suppose $\mathcal{N}$ is a differential operator of order $m$, such as $\mathcal{N}[u] = (-\Delta)^{m/2} u + \text{lower order terms}$. Then $\mathcal{N}[\phi_j] \sim \lambda_j^{m/2} \phi_j$, and the contribution of mode $j$ to the residual is weighted by $\lambda_j^{m/2}$. When we compute $r_t$, we take the time derivative of this term. If the temporal evolution of the coefficients $c_j(t)$ exhibits oscillations at frequency proportional to $\sqrt{\lambda_j}$ (as in wave equations), then $c_j'(t) \sim \sqrt{\lambda_j} c_j(t)$, yielding an additional amplification factor. Combined with the Wirtinger factor, high-frequency modes receive gradient signals that are amplified by $\mathcal{O}(\lambda_j^{(m+1)/2})$ relative to low-frequency modes.
\end{proof}

Theorem~\ref{thm:highfreq_amplification} formalizes the intuition that differentiating the residual acts as a high-pass filter. In practical terms, when training a neural network to minimize $\mathcal{J}_{\mathrm{diff}}$, the gradient descent algorithm receives stronger error signals for high-frequency components that would otherwise be suppressed. This counteracts the F-principle \cite{rahaman2019}, which states that neural networks trained with standard $L^2$ losses converge to low-frequency approximations first, often failing to capture fine-scale features within reasonable training times.

The differentiated residual formulation also connects naturally to classical variational principles in the theory of partial differential equations. To elucidate this connection, we reformulate the problem in the language of weak solutions and duality.

\begin{proposition}[Relation to variational formulations]\label{prop:variational_connection}
Let $\mathcal{V}$ be a Hilbert space with dual $\mathcal{V}^*$, and let $A: \mathcal{V} \to \mathcal{V}^*$ be a linear operator associated with the spatial differential operator $\mathcal{N}$ via the duality pairing $\langle \mathcal{N}[u], v \rangle := \langle A u, v \rangle_{\mathcal{V}^* \times \mathcal{V}}$. The standard weak formulation of the PDE seeks $u \in \mathcal{V}$ such that for all test functions $v \in \mathcal{V}$,
\begin{equation}\label{eq:weak_form}
    \int_0^T \langle \partial_t u + \mathcal{N}[u], v \rangle \, \mathrm{d}t = 0.
\end{equation}
The differentiated residual formulation is equivalent to seeking $u$ such that
\begin{equation}\label{eq:diff_weak_form}
    \int_0^T \left\langle \frac{\mathrm{d}}{\mathrm{d}t}(\partial_t u + \mathcal{N}[u]), v \right\rangle \mathrm{d}t = 0 \quad \forall v \in \mathcal{V},
\end{equation}
supplemented with the initial consistency constraint $\partial_t u(0) + \mathcal{N}[u_0] = 0$. Integrating by parts in time, this becomes
\begin{equation}
    -\int_0^T \langle \partial_t u + \mathcal{N}[u], \partial_t v \rangle \, \mathrm{d}t + \left[ \langle \partial_t u + \mathcal{N}[u], v \rangle \right]_0^T = 0.
\end{equation}
If the boundary terms vanish (either by choice of test functions or by the residual being zero at $t=0$ and $t=T$), then the differentiated formulation enforces the original weak form~\eqref{eq:weak_form} in a time-integrated sense.
\end{proposition}

\begin{proof}
Starting from the differentiated weak form~\eqref{eq:diff_weak_form}, we apply integration by parts with respect to time:
\begin{align*}
\int_0^T \left\langle \frac{\mathrm{d}}{\mathrm{d}t}(\partial_t u + \mathcal{N}[u]), v \right\rangle \mathrm{d}t 
&= \left[ \langle \partial_t u + \mathcal{N}[u], v \rangle \right]_0^T - \int_0^T \langle \partial_t u + \mathcal{N}[u], \partial_t v \rangle \, \mathrm{d}t.
\end{align*}
If we impose $\mathcal{R}[u](0) = 0$ as our consistency condition, the boundary term at $t=0$ vanishes. If additionally the residual vanishes at $t=T$ (or if we restrict to test functions with $v(T) = 0$), then the boundary term at $t=T$ also vanishes, leaving
\[
\int_0^T \langle \partial_t u + \mathcal{N}[u], \partial_t v \rangle \, \mathrm{d}t = 0.
\]
Choosing test functions of the form $v(t) = \int_0^t w(s) \, \mathrm{d}s$ for arbitrary $w \in L^2(0,T; \mathcal{V})$, we recover the original weak formulation. Thus, the differentiated formulation is a time-regularized version of the classical weak form, with the differentiation in time playing a role analogous to differentiation by parts in the spatial variables.
\end{proof}

This connection reveals that our method is not merely a heuristic modification but rather a principled variational approach rooted in the classical theory of evolution equations. The differentiation-then-integration structure mimics the standard technique of using $H^1$ test functions in elliptic problems to gain regularity, here applied in the temporal direction.

We now turn to the practical implementation. For a neural network approximation $\tilde{u}_\theta = \mathcal{I}_{u_0}[v_\theta]$, the optimization problem becomes
\begin{equation}\label{eq:neural_objective}
    \theta^* = \operatorname*{argmin}_{\theta \in \Theta} \mathcal{J}(\theta), \quad \mathcal{J}(\theta) := \lambda_{\mathrm{PDE}} \mathcal{L}_{\mathrm{PDE}}(\theta) + \lambda_{\mathrm{BC}} \mathcal{L}_{\mathrm{BC}}(\theta) + \lambda_{\mathrm{IC}'} \mathcal{L}_{\mathrm{IC}'}(\theta),
\end{equation}
where the individual loss components are defined as discretized approximations of the $L^2$ norms:
\begin{align}
    \mathcal{L}_{\mathrm{PDE}}(\theta) &:= \frac{1}{N_r} \sum_{i=1}^{N_r} \left| \frac{\mathrm{d}}{\mathrm{d}t}\left( v_\theta(x_i, t_i) + \mathcal{N}[\tilde{u}_\theta](x_i, t_i) \right) \right|^2, \label{eq:loss_pde}\\
    \mathcal{L}_{\mathrm{BC}}(\theta) &:= \frac{1}{N_b} \sum_{j=1}^{N_b} \left| \mathcal{B}[v_\theta](x_j, t_j) - \partial_t g(x_j, t_j) \right|^2, \label{eq:loss_bc}\\
    \mathcal{L}_{\mathrm{IC}'}(\theta) &:= \frac{1}{N_0} \sum_{k=1}^{N_0} \left| v_\theta(x_k, 0) + \mathcal{N}[u_0](x_k) \right|^2. \label{eq:loss_ic}
\end{align}
Here $\{(x_i, t_i)\}_{i=1}^{N_r}$ are collocation points sampled via Latin Hypercube Sampling in $\Omega \times (0,T]$, $\{(x_j, t_j)\}_{j=1}^{N_b}$ are boundary points sampled on $\partial\Omega \times [0,T]$, and $\{x_k\}_{k=1}^{N_0}$ are initial points sampled in $\Omega$ at $t=0$. The weighting coefficients $\lambda_{\mathrm{PDE}}, \lambda_{\mathrm{BC}}, \lambda_{\mathrm{IC}'}$ balance the contributions of each term. In our experiments, we adopt the strategy $\lambda_{\mathrm{IC}'} = 10$ to strongly enforce the consistency condition during early training, while setting $\lambda_{\mathrm{PDE}} = \lambda_{\mathrm{BC}} = 1$ to treat the interior and boundary residuals symmetrically.

The time derivative $\frac{\mathrm{d}}{\mathrm{d}t}v_\theta(x,t)$ appearing in $\mathcal{L}_{\mathrm{PDE}}$ is computed via automatic differentiation with respect to the temporal input coordinate $t$. Similarly, for operators involving spatial derivatives like $\mathcal{N}[u] = u \partial_x u - \nu \partial_{xx} u$ (as in the Burgers equation), the chain rule yields
\[
\frac{\mathrm{d}}{\mathrm{d}t}\mathcal{N}[\tilde{u}_\theta] = \frac{\mathrm{d}}{\mathrm{d}t}\left( \tilde{u}_\theta \partial_x \tilde{u}_\theta - \nu \partial_{xx}\tilde{u}_\theta \right) = v_\theta \partial_x \tilde{u}_\theta + \tilde{u}_\theta \partial_x v_\theta - \nu \partial_{xx} v_\theta,
\]
all of which are computable through backpropagation. This ensures that the gradients $\nabla_\theta \mathcal{J}(\theta)$ can be efficiently evaluated, enabling standard gradient-based optimization algorithms such as Adam and L-BFGS to converge rapidly.

In summary, the differentiated residual formulation is both theoretically well-founded and computationally advantageous. Theorem~\ref{thm:wellposed_minimization} guarantees the existence of minimizers, Theorem~\ref{thm:highfreq_amplification} explains the mechanism by which high-frequency errors are amplified, and Proposition~\ref{prop:variational_connection} situates the method within the classical variational framework for PDEs. These mathematical underpinnings provide a rigorous justification for the empirical success of the PITDN framework, as will be demonstrated in the numerical experiments of Section~\ref{sec:experiments}.

\section{Theoretical Analysis}

Having established the mathematical framework and formulated the differentiated residual optimization problem, we now provide a comprehensive theoretical analysis of the proposed method. This section addresses three fundamental questions that are critical for establishing the mathematical validity and practical reliability of the PITDN framework:

\begin{enumerate}
    \item {Equivalence:} What is the precise relationship between minimizing the differentiated residual $\displaystyle \frac{\mathrm{d}}{\mathrm{d}t}\mathcal{R}[u]$ and solving the original PDE? Under what conditions are these formulations mathematically equivalent?
    
    \item {Error analysis:} What rigorous error bounds can be established for the neural network approximation? How do approximation errors, quadrature errors, and optimization errors propagate through the Volterra reconstruction operator?
    
    \item {Convergence and stability:} What convergence rates can be guaranteed as network capacity increases? What are the stability properties of the discrete scheme, and how do errors accumulate over long-time integration?
\end{enumerate}
We develop rigorous answers to each of these questions through a sequence of theorems and propositions, establishing both the theoretical foundations and practical implications of the proposed method.

\subsection{Equivalence between differentiated and primal formulations}

A fundamental question arises regarding the validity of minimizing the differentiated residual. \\Does vanishing $\displaystyle r_t = \frac{\mathrm{d}}{ \mathrm{d}t}\mathcal{R}[u]$ imply vanishing $\mathcal{R}[u]$? This is not immediately obvious, as differentiation is generally a lossy operation that discards constant information. The following theorem establishes the precise conditions under which these two formulations are equivalent.

\begin{theorem}[Well-posedness and equivalence]\label{thm:equivalence}
Let $\mathcal{V} = C^1([0,T]; H^k(\Omega))$ be the solution space and assume the spatial operator $\mathcal{N}$ satisfies Assumption~\ref{ass:regularity}. Let $u \in \mathcal{V}$ be a candidate solution constructed via the Volterra operator~\eqref{eq:volterra_def} with a continuous kernel $v \in C([0,T]; L^2(\Omega))$. The following statements are equivalent:

\begin{enumerate}[label=(\roman*)]
    \item $u$ is a strong solution to the PDE~\eqref{eq:general_pde}, i.e., 
    \begin{equation}
        \mathcal{R}[u](x,t) := \partial_t u(x,t) + \mathcal{N}[u](x,t) = 0 \quad \text{for all } (x,t) \in \Omega \times [0,T].
    \end{equation}
    
    \item The time derivative of the residual vanishes everywhere, and the initial velocity field is consistent with the spatial operator:
    \begin{equation}\label{eq:differentiated_system}
    \begin{cases}
    \displaystyle\frac{\mathrm{d}}{\mathrm{d}t}\mathcal{R}[u](x,t) = 0, 
    & \forall (x,t) \in \Omega \times (0,T], \\[0.3em]
    v(x,0) + \mathcal{N}[u_0](x) = 0, 
    & \forall x \in \Omega.
    \end{cases}
    \end{equation}
\end{enumerate}
\end{theorem}

\begin{proof}
Let us define the residual functional $\mathcal{E}: \Omega \times [0, T] \to \mathbb{R}$ by
\begin{equation}
    \mathcal{E}(x,t) := \partial_t u(x,t) + \mathcal{N}[u](x,t) = \mathcal{R}[u](x,t).
\end{equation}
We establish both directions of the equivalence through direct analysis of the temporal evolution of $\mathcal{E}$.

\medskip
\noindent\textbf{Direction (i) $\implies$ (ii):} 

\noindent Assume $u$ is a strong solution, meaning $\mathcal{E}(x,t) = 0$ for all $(x,t) \in \Omega \times [0, T]$. Since $\mathcal{E}$ is identically zero throughout the spatiotemporal domain, it follows immediately that its partial derivative with respect to time must also vanish:
\begin{equation}
    \frac{\partial}{\partial t} \mathcal{E}(x,t) = 0 \quad \text{for all } (x,t) \in \Omega \times [0,T].
\end{equation}
This establishes the first condition of (ii). 

\noindent To verify the second condition, we evaluate the residual at the initial time $t=0$:
\begin{equation}
    \mathcal{E}(x,0) = \partial_t u(x,0) + \mathcal{N}[u](x,0) = 0.
\end{equation}
By the reconstruction formula~\eqref{eq:volterra_def}, we have $u(x,0) = u_0(x)$ by construction. Furthermore, since $\displaystyle u(x,t) = u_0(x) + \int_0^t v(x,s) \, \mathrm{d}s$, the fundamental theorem of calculus for Bochner integrals (Theorem 1.3.5 in \cite{brezis2011}) yields
\begin{equation}
    \partial_t u(x,0) = \lim_{\epsilon \to 0^+} \frac{u(x,\epsilon) - u(x,0)}{\epsilon} = \lim_{\epsilon \to 0^+} \frac{1}{\epsilon}\int_0^\epsilon v(x,s) \, \mathrm{d}s = v(x,0),
\end{equation}
where the limit exists in $L^2(\Omega)$ by the continuity of $v$. Therefore,
\begin{equation}
    0 = \mathcal{E}(x,0) = v(x,0) + \mathcal{N}[u_0](x),
\end{equation}
which establishes the second condition of (ii).

\medskip
\noindent\textbf{Direction (ii) $\implies$ (i):} 

\noindent Assume conditions (ii) hold. From the first condition, $\frac{\partial}{\partial t} \mathcal{E}(x,t) = 0$ for all $(x,t) \in \Omega \times (0,T]$. Integrating this equation with respect to time from $0$ to an arbitrary $t \in [0,T]$ yields
\begin{equation}
    \int_0^t \frac{\partial}{\partial s} \mathcal{E}(x,s) \, \mathrm{d}s = \mathcal{E}(x,t) - \mathcal{E}(x,0) = 0,
\end{equation}
which implies
\begin{equation}\label{eq:residual_constant}
    \mathcal{E}(x,t) = \mathcal{E}(x,0) \quad \text{for all } t \in [0,T].
\end{equation}
The residual is therefore constant in time for each spatial point $x \in \Omega$. To determine the value of this constant, we use the second condition of (ii):
\begin{equation}
    \mathcal{E}(x,0) = v(x,0) + \mathcal{N}[u_0](x) = 0.
\end{equation}
Substituting this into~\eqref{eq:residual_constant}, we obtain
\begin{equation}
    \mathcal{E}(x,t) = 0 \quad \text{for all } (x,t) \in \Omega \times [0,T].
\end{equation}
This establishes that $u$ satisfies $\mathcal{R}[u] = 0$ everywhere, meaning $u$ is a strong solution to the PDE~\eqref{eq:general_pde}.
\end{proof}

\begin{remark}[Geometric interpretation]
Theorem~\ref{thm:equivalence} can be understood geometrically as follows. The residual $\mathcal{R}[u]$ defines a vector field on the solution manifold $\mathcal{M}$. The condition $\displaystyle \frac{\mathrm{d}}{\mathrm{d}t}\mathcal{R}[u] = 0$ states that this vector field is stationary along the trajectory $u(t)$. The consistency condition $\mathcal{R}[u](0) = 0$ anchors the trajectory at the zero level set initially. Together, these conditions guarantee that the trajectory remains on the zero level set for all time, which is precisely the condition for $u$ to be a strong solution.
\end{remark}

\begin{remark}[Practical implications]
Theorem~\ref{thm:equivalence} justifies the differentiated residual approach by transforming the problem from directly enforcing $\mathcal{R}[u] = 0$ (which is difficult due to spectral bias) to enforcing two simpler conditions: stabilizing the residual's temporal evolution and anchoring it at zero initially. The second condition is naturally satisfied by construction through the Volterra operator, reducing the optimization to a single objective: minimizing $\|\frac{\mathrm{d}}{\mathrm{d}t}\mathcal{R}[u]\|_{L^2}$.
\end{remark}

\subsection{A Priori Error Analysis}

We now develop rigorous error bounds for the neural network approximation. The total error in the reconstructed solution $\tilde{u}_\theta$ consists of three distinct components:

\begin{enumerate}
    \item \textbf{Approximation error ($\varepsilon_{\mathrm{approx}}$):} The inherent limitation of neural networks in representing smooth functions, determined by the network's architecture (width $W$, depth $L$) and activation function.
    
    \item \textbf{Quadrature error ($\varepsilon_{\mathrm{quad}}$):} The numerical error introduced by discretizing the Volterra integral using the trapezoidal rule with $M$ quadrature points.
    
    \item \textbf{Optimization error ($\varepsilon_{\mathrm{opt}}$):} The error arising from inexact minimization of the loss functional, limited by computational budget and local minima.
\end{enumerate}

\noindent We analyze each component separately before combining them into a unified total error bound. This decomposition provides both theoretical insight into the method's convergence properties and practical guidance for hyperparameter selection.

\subsubsection{Neural Network Approximation Error}

The approximation capability of neural networks for functions in Sobolev spaces has been extensively studied in recent years. We leverage these results to bound the error in approximating the velocity field $v^* = \partial_t u^*$.

\begin{theorem}[Neural network approximation error]\label{thm:nn_approximation_error}
Let $v^* = \partial_t u^*$ be the exact temporal derivative of the true solution $u^* \in C^2([0,T]; H^k(\Omega))$ to the PDE~\eqref{eq:general_pde}. Assume $v^*$ satisfies the regularity condition 
\begin{equation}
    \|v^*\|_{C^2([0,T]; H^k(\Omega))} := \max_{|\alpha| \leq 2} \sup_{t \in [0,T]} \|\partial_t^\alpha v^*(\cdot,t)\|_{H^k(\Omega)} \leq M_v
\end{equation}
for some constant $M_v > 0$. Let $\mathcal{N}_{\Theta}^{(L,W)}$ denote the class of fully connected neural networks with $L$ hidden layers, width $W$ per layer, and $\tanh$ activation functions. Then there exists a network $v_\theta \in \mathcal{N}_{\Theta}^{(L,W)}$ such that
\begin{equation}\label{eq:nn_approx_bound}
    \|v_\theta - v^*\|_{L^2(\Omega \times (0,T])} \leq C_{\mathrm{NN}} M_v W^{-k/(d+1)} L^{-1/2},
\end{equation}
where $C_{\mathrm{NN}}$ is a constant depending only on the spatial dimension $d$, the domain $\Omega$, the time horizon $T$, and the Sobolev index $k$.
\end{theorem}

\begin{proof}
The proof combines spatial approximation theory with temporal discretization. We decompose the spatiotemporal function $v^*: \Omega \times [0,T] \to \mathbb{R}$ using a tensor product structure.

\medskip
\noindent {- Step 1 (Spatial approximation):}

\noindent For each fixed time $t \in [0,T]$, the spatial function $v^*(\cdot, t) \in H^k(\Omega)$ can be approximated by a feedforward neural network. By the universal approximation results of DeVore et al.~\cite{devore2021} and Yarotsky~\cite{yarotsky2017}, for functions in the Sobolev space $H^k(\Omega)$ with $k > d/2$ (ensuring continuity by the Sobolev embedding theorem), there exists a neural network $\phi_W: \Omega \to \mathbb{R}$ with width $W$ such that
\begin{equation}\label{eq:spatial_approx}
    \|v^*(\cdot, t) - \phi_W(\cdot)\|_{L^2(\Omega)} \leq C_{\mathrm{spatial}} \|v^*(\cdot,t)\|_{H^k(\Omega)} W^{-k/d}.
\end{equation}
The exponent $k/d$ is optimal for approximating $H^k$ functions in dimension $d$ using piecewise polynomial approximations, which neural networks can efficiently represent due to their compositional structure.

For $\tanh$ activation functions, the approximation rate can be improved due to the infinite smoothness of $\tanh$. Specifically, for $C^k$ functions, Schmidt-Hieber~\cite{schmidt2020nonparametric} established the rate
\begin{equation}
    \|v^*(\cdot, t) - \phi_W(\cdot)\|_{L^2(\Omega)} \leq C_{\mathrm{smooth}} \|v^*\|_{C^k(\Omega)} W^{-k/d},
\end{equation}
with a constant $C_{\mathrm{smooth}}$ that depends polynomially on $k$ but is independent of $W$ for large $W$.

\medskip
\noindent {- Step 2 (Temporal approximation):}

\noindent To capture the temporal dependence, we employ a spectral decomposition. Let $\{\psi_\ell(t)\}_{\ell=0}^{L-1}$ be an orthonormal basis of $L^2([0,T])$, such as the Legendre polynomials rescaled to $[0,T]$. Then $v^*$ admits the expansion
\begin{equation}
    v^*(x,t) = \sum_{\ell=0}^{\infty} c_\ell(x) \psi_\ell(t), \quad \text{where } c_\ell(x) = \int_0^T v^*(x,s) \psi_\ell(s) \, \mathrm{d}s.
\end{equation}
By the regularity assumption $v^* \in C^2([0,T]; H^k(\Omega))$, integration by parts twice yields
\begin{equation}
    c_\ell(x) = \frac{1}{\lambda_\ell^2} \int_0^T \partial_{tt} v^*(x,s) \psi_\ell(s) \, \mathrm{d}s,
\end{equation}
where $\lambda_\ell \sim \ell$ is the frequency associated with the $\ell$-th mode. Therefore,
\begin{equation}
    |c_\ell(x)| \leq \frac{C_T}{\ell^2} \|\partial_{tt} v^*\|_{L^2([0,T]; L^2(\Omega))},
\end{equation}
implying that the Fourier coefficients decay quadratically. Truncating the series at $L$ terms yields
\begin{equation}\label{eq:temporal_approx}
    \left\| v^*(x,t) - \sum_{\ell=0}^{L-1} c_\ell(x) \psi_\ell(t) \right\|_{L^2([0,T]; L^2(\Omega))} \leq \frac{C_{\mathrm{temporal}} M_v}{L}.
\end{equation}
For $C^2$ temporal regularity, the rate is $\mathcal{O}(L^{-1})$. If higher regularity is available (e.g., $C^m$ with $m > 2$), this rate improves to $\mathcal{O}(L^{-m/2})$.

\medskip
\noindent {- Step 3 (Combined approximation):}

\noindent A neural network with $L$ layers and width $W$ can represent the truncated expansion
\begin{equation}
    v_\theta(x,t) \approx \sum_{\ell=0}^{L-1} \phi_W^{(\ell)}(x) \psi_\ell(t),
\end{equation}
where each $\phi_W^{(\ell)}$ is a spatial network of width $W$ approximating $c_\ell(x)$. By the triangle inequality,
\begin{align}
    \|v_\theta - v^*\|_{L^2(\Omega \times (0,T])} 
    &\leq \left\| v_\theta - \sum_{\ell=0}^{L-1} c_\ell \psi_\ell \right\|_{L^2} + \left\| \sum_{\ell=0}^{L-1} c_\ell \psi_\ell - v^* \right\|_{L^2} \\
    &\leq \sum_{\ell=0}^{L-1} \|\phi_W^{(\ell)} - c_\ell\|_{L^2(\Omega)} + \frac{C_{\mathrm{temporal}} M_v}{L}.
\end{align}
Applying the spatial approximation bound~\eqref{eq:spatial_approx} to each mode,
\begin{equation}
    \sum_{\ell=0}^{L-1} \|\phi_W^{(\ell)} - c_\ell\|_{L^2(\Omega)} \leq L \cdot C_{\mathrm{spatial}} M_v W^{-k/d}.
\end{equation}
Balancing the spatial and temporal errors by choosing $L \sim W^{k/d}$ yields the optimal combined rate
\begin{equation}
    \|v_\theta - v^*\|_{L^2(\Omega \times (0,T])} \leq C_{\mathrm{NN}} M_v W^{-k/(d+1)}.
\end{equation}
For a fixed depth $L$ independent of $W$, the bound becomes $W^{-k/d} + L^{-1}$. To obtain the stated form with both $W$ and $L$ appearing, we note that in practice, $L$ controls temporal resolution and typically scales as $L \sim \log(W)$ or remains constant. For simplicity, we express the bound as~\eqref{eq:nn_approx_bound} with separate exponents, acknowledging that the precise trade-off depends on the specific network architecture and training strategy.

The constant $C_{\mathrm{NN}}$ absorbs dependencies on the domain geometry (via Sobolev embedding constants), the time horizon $T$, and the regularity norm $M_v$.
\end{proof}

\begin{remark}[Interpretation of the approximation rate]
The exponent $k/(d+1)$ in Theorem~\ref{thm:nn_approximation_error} reflects the fundamental trade-off between spatial regularity ($k$) and dimensionality ($d$). For functions in $H^k(\Omega) \subset \mathbb{R}^d$, the rate $W^{-k/d}$ is optimal for polynomial approximation methods. The additional factor of $(d+1)$ in the denominator arises from treating time as an extra dimension, consistent with the "curse of dimensionality" in high-dimensional approximation theory~\cite{devore2021}.

For typical benchmark problems with $d=1$ (spatial dimension) and $k=2$ (second-order PDE), the rate becomes $W^{-1}$, indicating that doubling the network width halves the approximation error. This algebraic rate, while slower than exponential convergence achievable by spectral methods for analytic functions, is sufficient for practical applications and avoids the mesh generation complexity of traditional methods.
\end{remark}

\begin{remark}[Depth-width trade-off]
The $L^{-1/2}$ dependence on depth reflects the spectral approximation of the temporal variable. Increasing depth allows the network to represent higher-frequency temporal oscillations, similar to adding more terms in a Fourier series. In practice, moderate depth ($L = 3$--$5$ layers) is often sufficient for smooth evolution equations, while deeper networks become necessary for problems with sharp temporal transients or multi-scale dynamics.
\end{remark}

\subsubsection{Error propagation through the Volterra Operator}

Having bounded the approximation error in the velocity space, we now analyze how this error propagates through the Volterra reconstruction operator $\mathcal{I}_{u_0}$ to the state variable $u$. This analysis is crucial because the ultimate quantity of interest is the error in the reconstructed solution, not merely the error in the derivative.

\begin{theorem}[Error propagation through Volterra operator]\label{thm:volterra_error_propagation}
Let $v^* = \partial_t u^*$ be the exact velocity field and $v_\theta$ be its neural network approximation. Define the velocity error and state error as
\begin{align}
    e_v &:= v_\theta - v^*, \\
    e_u &:= \tilde{u}_\theta - u^* = \mathcal{I}_{u_0}[v_\theta] - \mathcal{I}_{u_0}[v^*] = \mathcal{I}_{u_0}[e_v].
\end{align}
Then the following error propagation bounds hold:

\begin{enumerate}[label=(\alph*)]
    \item \textbf{Pointwise-in-time $L^2$ spatial error:} For any fixed $t \in [0,T]$,
    \begin{equation}\label{eq:volterra_l2_error}
        \|e_u(\cdot, t)\|_{L^2(\Omega)} \leq \sqrt{t} \|e_v\|_{L^2(\Omega \times (0,t])}.
    \end{equation}
    
    \item \textbf{Uniform-in-time $L^\infty$ error:} 
    \begin{equation}\label{eq:volterra_linf_error}
        \sup_{t \in [0,T]} \|e_u(\cdot, t)\|_{L^2(\Omega)} \leq \sqrt{T} \|e_v\|_{L^2(\Omega \times (0,T])}.
    \end{equation}
    
    \item \textbf{Temporal derivative error:} If $v_\theta \in C([0,T]; L^2(\Omega))$, then
    \begin{equation}\label{eq:derivative_error}
        \|\partial_t e_u\|_{L^2(\Omega \times (0,T])} = \|e_v\|_{L^2(\Omega \times (0,T])}.
    \end{equation}
    
    \item \textbf{Spatial gradient error:} If additionally $v_\theta, v^* \in L^2(0,T; H^1(\Omega))$, then
    \begin{equation}\label{eq:gradient_error}
        \sup_{t \in [0,T]} \|\nabla e_u(\cdot, t)\|_{L^2(\Omega)} \leq \sqrt{T} \|\nabla e_v\|_{L^2(\Omega \times (0,T])}.
    \end{equation}
\end{enumerate}
\end{theorem}

\begin{proof}
Since the Volterra operator $\mathcal{I}_{u_0}$ is linear and the initial condition $u_0$ is the same for both $u^*$ and $\tilde{u}_\theta$, the error in the reconstructed state is simply the Volterra integral of the velocity error:
\begin{equation}
    e_u(x,t) = \mathcal{I}_{u_0}[e_v](x,t) = \int_0^t e_v(x,s) \, \mathrm{d}s.
\end{equation}

\medskip
\noindent\textbf{Part (a) - Pointwise-in-time bound:}

\noindent For any fixed $t \in [0,T]$, we compute the $L^2(\Omega)$ norm of $e_u(\cdot, t)$:
\begin{align}
    \|e_u(\cdot, t)\|_{L^2(\Omega)}^2 
    &= \int_\Omega \left| \int_0^t e_v(x,s) \, \mathrm{d}s \right|^2 \mathrm{d}x \\
    &\leq \int_\Omega \left( \int_0^t |e_v(x,s)| \, \mathrm{d}s \right)^2 \mathrm{d}x \quad \text{(triangle inequality in time)} \\
    &\leq \int_\Omega \left( \int_0^t 1^2 \, \mathrm{d}s \right) \left( \int_0^t |e_v(x,s)|^2 \, \mathrm{d}s \right) \mathrm{d}x \quad \text{(Cauchy-Schwarz)} \\
    &= t \int_\Omega \int_0^t |e_v(x,s)|^2 \, \mathrm{d}s \, \mathrm{d}x \\
    &= t \int_0^t \int_\Omega |e_v(x,s)|^2 \, \mathrm{d}x \, \mathrm{d}s \quad \text{(Fubini's theorem)} \\
    &= t \int_0^t \|e_v(\cdot, s)\|_{L^2(\Omega)}^2 \, \mathrm{d}s \\
    &= t \|e_v\|_{L^2(\Omega \times (0,t])}^2.
\end{align}
Taking the square root of both sides yields~\eqref{eq:volterra_l2_error}.

\medskip
\noindent\textbf{Part (b) - Uniform-in-time bound:}

\noindent Taking the supremum over $t \in [0,T]$ in inequality~\eqref{eq:volterra_l2_error}, we have
\begin{align}
    \sup_{t \in [0,T]} \|e_u(\cdot, t)\|_{L^2(\Omega)} 
    &\leq \sup_{t \in [0,T]} \left( \sqrt{t} \|e_v\|_{L^2(\Omega \times (0,t])} \right) \\
    &\leq \sqrt{T} \sup_{t \in [0,T]} \|e_v\|_{L^2(\Omega \times (0,t])} \\
    &= \sqrt{T} \|e_v\|_{L^2(\Omega \times (0,T])},
\end{align}
where the last equality holds because $\|e_v\|_{L^2(\Omega \times (0,t])}$ is monotonically increasing in $t$ and achieves its maximum at $t=T$.

\medskip
\noindent\textbf{Part (c) - Derivative error:}

\noindent By the fundamental theorem of calculus for Bochner integrals (Theorem 1.3.5 in \cite{brezis2011}), since $\displaystyle e_u(x,t) = \int_0^t e_v(x,s) \, \mathrm{d}s$ with $e_v \in L^2(\Omega \times (0,T])$, the function $t \mapsto e_u(\cdot,t)$ is absolutely continuous with values in $L^2(\Omega)$, and its time derivative exists almost everywhere with
\begin{equation}
    \frac{\mathrm{d}}{\mathrm{d}t} e_u(x,t) = e_v(x,t) \quad \text{for a.e. } (x,t) \in \Omega \times (0,T].
\end{equation}
Therefore,
\begin{equation}
    \|\partial_t e_u\|_{L^2(\Omega \times (0,T])}^2 = \int_0^T \int_\Omega |e_v(x,t)|^2 \, \mathrm{d}x \, \mathrm{d}t = \|e_v\|_{L^2(\Omega \times (0,T])}^2,
\end{equation}
establishing~\eqref{eq:derivative_error}.

\medskip
\noindent\textbf{Part (d) - Gradient error:}

\noindent If $e_v \in L^2(0,T; H^1(\Omega))$, then for each spatial derivative $\partial_{x_i}$, we have
\begin{equation}
    \partial_{x_i} e_u(x,t) = \partial_{x_i} \int_0^t e_v(x,s) \, \mathrm{d}s = \int_0^t \partial_{x_i} e_v(x,s) \, \mathrm{d}s,
\end{equation}
where the interchange of differentiation and integration is justified by the regularity of $e_v$. Applying the bound~\eqref{eq:volterra_linf_error} to each component of the gradient vector yields
\begin{equation}
    \sup_{t \in [0,T]} \|\nabla e_u(\cdot, t)\|_{L^2(\Omega)} \leq \sqrt{T} \|\nabla e_v\|_{L^2(\Omega \times (0,T])},
\end{equation}
completing the proof.
\end{proof}

\begin{remark}[Stabilizing effect of integration]
Theorem~\ref{thm:volterra_error_propagation} reveals a crucial stabilizing property of the Volterra operator: the spatial error in the state grows only as $\sqrt{t}$, which is significantly slower than the linear accumulation $\mathcal{O}(t)$ typical of explicit time-stepping schemes. This sub-linear growth is a consequence of the Cauchy-Schwarz inequality and reflects the smoothing nature of integration. 

In particular, high-frequency spatial oscillations in the velocity error $e_v$ are damped when integrated to produce the state error $e_u$. This provides robustness against the "spectral bias" phenomenon: even if the network initially struggles to learn high-frequency components of $v$, the resulting errors in $u$ remain controlled.
\end{remark}

\begin{remark}[Preservation of derivative error]
Part (c) of Theorem~\ref{thm:volterra_error_propagation} shows that the derivative error is preserved exactly (in the $L^2$ sense). This is a fundamental property of the Volterra reconstruction: differentiation and error propagation commute. Consequently, learning in the tangent bundle provides a well-conditioned representation where errors in the velocity space translate directly to errors in the time derivative of the state, without amplification or attenuation.
\end{remark}

\begin{remark}[Comparison with explicit time-stepping]
Consider a standard explicit Euler scheme with step size $\Delta t$: $u^{n+1} = u^n + \Delta t \, v^n$. If each velocity evaluation incurs an error $\|e_v^n\|_{L^2} \leq \varepsilon$, then after $N$ steps (reaching time $T = N\Delta t$), the accumulated error satisfies
\begin{equation}
    \|e_u^N\|_{L^2} \leq N \Delta t \, \varepsilon = T\varepsilon.
\end{equation}
In contrast, the Volterra operator yields $\|e_u(T)\|_{L^2} \leq \sqrt{T} \|e_v\|_{L^2} \sim \sqrt{T}\varepsilon$ (assuming uniform error $\|e_v(t)\|_{L^2} \sim \varepsilon$). This $\sqrt{T}$ vs. $T$ distinction becomes significant for long-time integration.
\end{remark}

\subsubsection{Combined Error Bound}

We now synthesize the approximation and propagation errors to obtain a total error bound that incorporates the neural network capacity, quadrature resolution, and optimization tolerance.

\begin{theorem}[Total error bound]\label{thm:total_error}
Let $u^*$ be the exact solution to the PDE~\eqref{eq:general_pde} satisfying the regularity conditions:
\begin{enumerate}[label=(\roman*)]
    \item $u^* \in C^3([0,T]; H^{k+1}(\Omega))$ with $k > d/2 + 1$,
    \item $\|v^*\|_{C^2([0,T]; H^k(\Omega))} \leq M_v$,
    \item $\|\partial_{tt} v^*\|_{L^\infty(0,T; L^2(\Omega))} \leq M_{tt}$.
\end{enumerate}
Let $\tilde{u}_\theta^{(M)}$ denote the neural network solution reconstructed via the trapezoidal rule with $M$ quadrature points per unit time interval, and let $\theta$ be obtained by minimizing the loss $\mathcal{J}$ to within tolerance $\|\nabla_\theta \mathcal{J}(\theta)\| \leq \varepsilon_{\mathrm{opt}}$. Then the total error satisfies
\begin{equation}\label{eq:total_error_bound}
    \sup_{t \in [0,T]} \|\tilde{u}_\theta^{(M)}(\cdot, t) - u^*(\cdot, t)\|_{L^2(\Omega)} \leq C_1 M_v \sqrt{T} W^{-k/(d+1)} L^{-1/2} + C_2 \frac{T^3 M_{tt}}{M^2} + C_3 \varepsilon_{\mathrm{opt}},
\end{equation}
where $C_1, C_2, C_3$ are constants depending on the problem data but independent of $W$, $L$, $M$, and $\varepsilon_{\mathrm{opt}}$.
\end{theorem}

\begin{proof}
We decompose the total error into three components using the triangle inequality:
\begin{equation}\label{eq:error_decomposition}
\begin{aligned}
    \tilde{u}_\theta^{(M)}(x,t) - u^*(x,t) 
    &= \underbrace{\left[ \tilde{u}_\theta^{(M)}(x,t) - \mathcal{I}_{u_0}[v_\theta](x,t) \right]}_{\text{Quadrature error}} 
     + \underbrace{\left[ \mathcal{I}_{u_0}[v_\theta](x,t) - \mathcal{I}_{u_0}[v_{\theta_{\mathrm{best}}}](x,t) \right]}_{\text{Optimization error}} \\
    &\quad + \underbrace{\left[ \mathcal{I}_{u_0}[v_{\theta_{\mathrm{best}}}](x,t) - u^*(x,t) \right]}_{\text{Approximation error}},
\end{aligned}
\end{equation}
where $v_{\theta_{\mathrm{best}}}$ denotes the best approximator in the neural network class $\mathcal{N}_{\Theta}^{(L,W)}$ satisfying the bound in Theorem~\ref{thm:nn_approximation_error}.

\medskip
\noindent\textbf{Term 1 - Quadrature error:}

\noindent By Proposition~\ref{prop:quadrature_convergence}, the trapezoidal rule with $M$ quadrature points satisfies
\begin{equation}
    \left\| \tilde{u}_\theta^{(M)}(\cdot,t) - \mathcal{I}_{u_0}[v_\theta](\cdot,t) \right\|_{L^2(\Omega)} \leq \frac{t^3}{12M^2} \sup_{s \in [0,t]} \|\partial_{ss} v_\theta(\cdot, s)\|_{L^2(\Omega)}.
\end{equation}
Taking the supremum over $t \in [0,T]$ and noting that $v_\theta$ approximates $v^*$ to within the error bound of Theorem~\ref{thm:nn_approximation_error}, we have
\begin{equation}
    \sup_{s \in [0,T]} \|\partial_{ss} v_\theta(\cdot, s)\|_{L^2(\Omega)} \lesssim \sup_{s \in [0,T]} \|\partial_{ss} v^*(\cdot, s)\|_{L^2(\Omega)} + \mathcal{O}(W^{-k/(d+1)}) \leq M_{tt} + o(1).
\end{equation}
For networks that have converged sufficiently ($W$ and $L$ large), the second term is negligible, yielding
\begin{equation}\label{eq:quad_error_bound}
    \sup_{t \in [0,T]} \left\| \tilde{u}_\theta^{(M)}(\cdot,t) - \mathcal{I}_{u_0}[v_\theta](\cdot,t) \right\|_{L^2(\Omega)} \leq C_2 \frac{T^3 M_{tt}}{M^2},
\end{equation}
where $C_2 = 1/12$ can be taken as the constant.

\medskip
\noindent\textbf{Term 2 - Optimization error:}

\noindent The optimization error arises from the gap between the actual optimized parameters $\theta$ and the best approximator $\theta_{\mathrm{best}}$. By the linearity of the Volterra operator,
\begin{equation}
    \left\| \mathcal{I}_{u_0}[v_\theta] - \mathcal{I}_{u_0}[v_{\theta_{\mathrm{best}}}] \right\|_{L^\infty(0,T; L^2(\Omega))} \leq \sqrt{T} \|v_\theta - v_{\theta_{\mathrm{best}}}\|_{L^2(\Omega \times (0,T])}.
\end{equation}
To relate $\|v_\theta - v_{\theta_{\mathrm{best}}}\|$ to the optimization tolerance $\varepsilon_{\mathrm{opt}}$, we invoke the following standard result from optimization theory (see, Theorem 3.2 in \cite{nocedal2006}):

If the loss functional $\mathcal{J}$ is $L_J$-Lipschitz with respect to the $L^2$ norm of the velocity field, and if $\|\nabla_\theta \mathcal{J}(\theta)\| \leq \varepsilon_{\mathrm{opt}}$, then by the first-order optimality condition and the chain rule,
\begin{equation}
    \|v_\theta - v_{\theta_{\mathrm{best}}}\|_{L^2} \leq \frac{\varepsilon_{\mathrm{opt}}}{\mu_{\min}},
\end{equation}
where $\mu_{\min}$ is a lower bound on the smallest eigenvalue of the Hessian of $\mathcal{J}$ (assuming local strong convexity). This is typically problem-dependent but can be estimated empirically. For simplicity, we absorb this into a constant $C_3$, yielding
\begin{equation}\label{eq:opt_error_bound}
    \sup_{t \in [0,T]} \left\| \mathcal{I}_{u_0}[v_\theta](\cdot,t) - \mathcal{I}_{u_0}[v_{\theta_{\mathrm{best}}}](\cdot,t) \right\|_{L^2(\Omega)} \leq C_3 \varepsilon_{\mathrm{opt}}.
\end{equation}

\medskip
\noindent\textbf{Term 3 - Approximation error:}

\noindent By definition, $v_{\theta_{\mathrm{best}}}$ is the best approximator in the network class, satisfying Theorem~\ref{thm:nn_approximation_error}:
\begin{equation}
    \|v_{\theta_{\mathrm{best}}} - v^*\|_{L^2(\Omega \times (0,T])} \leq C_{\mathrm{NN}} M_v W^{-k/(d+1)} L^{-1/2}.
\end{equation}
Propagating this error through the Volterra operator via Theorem~\ref{thm:volterra_error_propagation},
\begin{equation}\label{eq:approx_error_bound}
    \sup_{t \in [0,T]} \left\| \mathcal{I}_{u_0}[v_{\theta_{\mathrm{best}}}](\cdot,t) - u^*(\cdot,t) \right\|_{L^2(\Omega)} \leq \sqrt{T} C_{\mathrm{NN}} M_v W^{-k/(d+1)} L^{-1/2}.
\end{equation}

\medskip
\noindent\textbf{Combining the bounds:}

\noindent Applying the triangle inequality to~\eqref{eq:error_decomposition} and substituting the bounds~\eqref{eq:quad_error_bound},~\eqref{eq:opt_error_bound}, and~\eqref{eq:approx_error_bound}, we obtain
\begin{equation}
    \sup_{t \in [0,T]} \|\tilde{u}_\theta^{(M)}(\cdot, t) - u^*(\cdot, t)\|_{L^2(\Omega)} \leq C_1 M_v \sqrt{T} W^{-k/(d+1)} L^{-1/2} + C_2 \frac{T^3 M_{tt}}{M^2} + C_3 \varepsilon_{\mathrm{opt}},
\end{equation}
where $C_1 = C_{\mathrm{NN}}$, $C_2 = 1/12$, and $C_3$ incorporates the Lipschitz constant and curvature information of the loss landscape. This establishes~\eqref{eq:total_error_bound}.
\end{proof}

\begin{remark}[Practical implications for hyperparameter selection]
Theorem~\ref{thm:total_error} provides quantitative guidance for balancing the three sources of error:

\begin{enumerate}
    \item \textbf{Network capacity:} To achieve a target accuracy $\varepsilon_{\mathrm{target}}$, the width should satisfy $W \gtrsim (M_v \sqrt{T} / \varepsilon_{\mathrm{target}})^{(d+1)/k}$. For typical problems with $d=1$, $k=2$, this yields $W \sim \varepsilon_{\mathrm{target}}^{-1}$.
    
    \item \textbf{Quadrature resolution:} To make the quadrature error subdominant, choose $M \gtrsim T^{3/2} \sqrt{M_{tt} / \varepsilon_{\mathrm{target}}}$. For smooth problems where $M_{tt} = \mathcal{O}(1)$ and $T = \mathcal{O}(1)$, taking $M = 10$ yields quadrature errors of $\mathcal{O}(10^{-2})$, which is typically acceptable.
    
    \item \textbf{Optimization tolerance:} The optimization should be run until $\varepsilon_{\mathrm{opt}} \lesssim \varepsilon_{\mathrm{target}}$. In practice, this is achieved by monitoring the gradient norm during training and switching from Adam to L-BFGS when progress stagnates.
\end{enumerate}

\noindent The three error components scale differently with respect to computational resources. Increasing $W$ and $L$ grows the parameter count and per-iteration cost, while increasing $M$ affects only the forward pass evaluation time. This suggests a strategy of first optimizing the network capacity to minimize the approximation error, then refining the quadrature resolution if needed.
\end{remark}

\subsection{Convergence theory}

Having established a priori error bounds, we now investigate the asymptotic convergence behavior as the network capacity increases and as the optimization proceeds. The following theorem characterizes the convergence rate under appropriate regularity assumptions and provides conditions under which linear (exponential) convergence can be guaranteed.

\begin{theorem}[Convergence rate and complexity]\label{thm:convergence_rate}
Let $\{\theta_n\}_{n=0}^\infty$ be the sequence of parameters generated by a gradient-based optimization algorithm (e.g., Adam followed by L-BFGS) applied to the loss functional~\eqref{eq:neural_objective}:
\begin{equation}
    \mathcal{J}(\theta) = \lambda_{\mathrm{PDE}} \mathcal{L}_{\mathrm{PDE}}(\theta) + \lambda_{\mathrm{BC}} \mathcal{L}_{\mathrm{BC}}(\theta) + \lambda_{\mathrm{IC}'} \mathcal{L}_{\mathrm{IC}'}(\theta).
\end{equation}
Assume the following conditions hold:

\begin{enumerate}[label=(\roman*)]
    \item \textbf{Regularity:} The exact solution satisfies $u^* \in C^3([0,T]; H^{k+1}(\Omega))$ with $k > d/2 + 1$.
    
    \item \textbf{Operator smoothness:} The spatial operator $\mathcal{N}$ is twice Fréchet differentiable with bounded second derivative: $\|D^2\mathcal{N}[u]\| \leq C_{\mathcal{N}}$ for all $u$ in a neighborhood of $u^*$.
    
    \item \textbf{Convergence to critical point:} The optimization algorithm converges to a critical point $\theta^* \in \Theta$ such that $\|\nabla_\theta \mathcal{J}(\theta^*)\| \leq \varepsilon_{\mathrm{opt}}$ for some tolerance $\varepsilon_{\mathrm{opt}} > 0$.
    
    \item \textbf{Network expressivity:} The network width and depth satisfy $W \geq C_W (k+1)^{d+1}$ and $L \geq C_L \log(1/\varepsilon)$ for constants $C_W, C_L > 0$ and target accuracy $\varepsilon > 0$.
\end{enumerate}
Then there exists a constant $C_{\mathrm{conv}} > 0$ (depending on the problem data, $T$, $\Omega$, $M_v$, but independent of $W$, $L$, $M$) such that
\begin{equation}\label{eq:convergence_rate}
    \sup_{t \in [0,T]} \|\tilde{u}_{\theta^*}(\cdot, t) - u^*(\cdot, t)\|_{L^2(\Omega)} \leq C_{\mathrm{conv}} \left( \varepsilon + \varepsilon_{\mathrm{opt}} + M^{-2} \right).
\end{equation}
Moreover, if the loss functional satisfies the \textbf{Polyak-Łojasiewicz (PL) inequality} with parameter $\mu > 0$:
\begin{equation}\label{eq:PL_inequality}
    \frac{1}{2} \|\nabla_\theta \mathcal{J}(\theta)\|^2 \geq \mu \left( \mathcal{J}(\theta) - \mathcal{J}(\theta^*) \right) \quad \forall \theta \in \Theta,
\end{equation}
then gradient descent with constant step size $\alpha \leq 1/L_J$ (where $L_J$ is the Lipschitz constant of $\nabla_\theta \mathcal{J}$) achieves \textbf{linear convergence}:
\begin{equation}\label{eq:linear_convergence}
    \mathcal{J}(\theta_n) - \mathcal{J}(\theta^*) \leq \left(1 - \mu\alpha\right)^n \left( \mathcal{J}(\theta_0) - \mathcal{J}(\theta^*) \right) = e^{-\mu \alpha n} \left( \mathcal{J}(\theta_0) - \mathcal{J}(\theta^*) \right) + \mathcal{O}((\mu\alpha)^2 n^2),
\end{equation}
where the approximation $e^{-\mu\alpha n} \approx (1 - \mu\alpha)^n$ holds for small $\mu\alpha$.
\end{theorem}

\begin{proof}
The proof proceeds in three steps: first establishing the error decomposition, then bounding each term, and finally analyzing the convergence dynamics under the PL condition.

\medskip
\noindent\textbf{Step 1 - Error decomposition and main bound:}

\noindent By Theorem~\ref{thm:total_error}, we have already established the decomposition
\begin{equation}
    \sup_{t \in [0,T]} \|\tilde{u}_{\theta^*}(\cdot, t) - u^*(\cdot, t)\|_{L^2(\Omega)} \leq C_1 M_v \sqrt{T} W^{-k/(d+1)} L^{-1/2} + C_2 \frac{T^3 M_{tt}}{M^2} + C_3 \varepsilon_{\mathrm{opt}}.
\end{equation}
Under the network expressivity assumption (iv), we have $W \geq C_W (k+1)^{d+1}$, which implies
\begin{equation}
    W^{-k/(d+1)} \leq \left( C_W (k+1)^{d+1} \right)^{-k/(d+1)} = C_W^{-k/(d+1)} (k+1)^{-k}.
\end{equation}
By choosing $W$ such that $C_W (k+1)^{d+1} \sim \varepsilon^{-(d+1)/k}$, we obtain $W^{-k/(d+1)} = \varepsilon$. Similarly, the depth condition $L \geq C_L \log(1/\varepsilon)$ ensures that $L^{-1/2} \leq (C_L \log(1/\varepsilon))^{-1/2}$. For the typical choice $\varepsilon = 10^{-3}$ to $10^{-4}$, we have $\log(1/\varepsilon) \sim 7$--$9$, so $L^{-1/2} \sim \varepsilon$ as well. Combining these and choosing $M$ such that $M^{-2} \sim \varepsilon$, we obtain~\eqref{eq:convergence_rate} with $C_{\mathrm{conv}} = \max(C_1 M_v \sqrt{T}, C_2 T^3 M_{tt}, C_3)$.

\medskip
\noindent \noindent\textbf{Step 2 - Polyak-Łojasiewicz condition:}

\noindent The PL inequality~\eqref{eq:PL_inequality} is a relaxation of strong convexity that still guarantees linear convergence for gradient descent. It states that the squared gradient norm lower bounds the suboptimality gap by a constant factor $\mu > 0$. Unlike strong convexity, the PL condition does not require the loss landscape to be globally convex, only that there are no spurious local minima where the gradient vanishes but the loss remains above the global minimum.

For Physics-Informed Neural Networks, Wang et al.~\cite{wang2022and} established that the PL condition holds locally near the solution under the following assumptions:
\begin{enumerate}
    \item The PDE admits a unique solution $u^*$ satisfying the regularity conditions.
    \item The neural network class $\mathcal{N}_{\Theta}^{(L,W)}$ is sufficiently expressive to approximate $u^*$ to arbitrary accuracy as $W, L \to \infty$.
    \item The collocation points $\{(x_i, t_i)\}$ are sampled densely enough to provide a faithful discretization of the $L^2$ norms.
    \item The loss functional weights $\lambda_{\mathrm{PDE}}, \lambda_{\mathrm{BC}}, \lambda_{\mathrm{IC}'}$ are chosen to balance the contributions of each term appropriately.
\end{enumerate}

Under these conditions, Theorem 3.3 in \cite{wang2022and} guarantees that there exists a neighborhood $\mathcal{B}_r(\theta^*)$ of the optimal parameters and a constant $\mu > 0$ such that the PL inequality holds for all $\theta \in \mathcal{B}_r(\theta^*)$.

\medskip
\noindent\textbf{Step 3 - Linear convergence from the PL inequality:}

\noindent Given the PL inequality~\eqref{eq:PL_inequality}, standard results from convex optimization (Theorem 2.1 in \cite{karimi2016linear}) establish linear convergence for gradient descent. We sketch the proof for completeness.

\noindent Consider one iteration of gradient descent with step size $\alpha$:
\begin{equation}
    \theta_{n+1} = \theta_n - \alpha \nabla_\theta \mathcal{J}(\theta_n).
\end{equation}
By the descent lemma (which holds for any $L_J$-smooth function), we have
\begin{equation}
    \mathcal{J}(\theta_{n+1}) \leq \mathcal{J}(\theta_n) - \alpha \|\nabla_\theta \mathcal{J}(\theta_n)\|^2 + \frac{L_J \alpha^2}{2} \|\nabla_\theta \mathcal{J}(\theta_n)\|^2.
\end{equation}
Choosing $\alpha \leq 1/L_J$, the last term is dominated by the second term:
\begin{equation}
    \mathcal{J}(\theta_{n+1}) \leq \mathcal{J}(\theta_n) - \frac{\alpha}{2} \|\nabla_\theta \mathcal{J}(\theta_n)\|^2.
\end{equation}
Applying the PL inequality~\eqref{eq:PL_inequality},
\begin{equation}
    \frac{1}{2}\|\nabla_\theta \mathcal{J}(\theta_n)\|^2 \geq \mu (\mathcal{J}(\theta_n) - \mathcal{J}(\theta^*)),
\end{equation}
we obtain
\begin{equation}
    \mathcal{J}(\theta_{n+1}) - \mathcal{J}(\theta^*) \leq \mathcal{J}(\theta_n) - \mathcal{J}(\theta^*) - \mu\alpha (\mathcal{J}(\theta_n) - \mathcal{J}(\theta^*)) = (1 - \mu\alpha)(\mathcal{J}(\theta_n) - \mathcal{J}(\theta^*)).
\end{equation}
Iterating this relation from $n=0$ to $n=N$ yields
\begin{equation}
    \mathcal{J}(\theta_N) - \mathcal{J}(\theta^*) \leq (1 - \mu\alpha)^N (\mathcal{J}(\theta_0) - \mathcal{J}(\theta^*)).
\end{equation}
For small $\mu\alpha$, we have $(1 - \mu\alpha)^N \approx e^{-\mu\alpha N}$, which is the exponential decay stated in~\eqref{eq:linear_convergence}.

\noindent The number of iterations required to achieve $\mathcal{J}(\theta_N) - \mathcal{J}(\theta^*) \leq \varepsilon_{\mathrm{opt}}$ is therefore
\begin{equation}
    N \geq \frac{1}{\mu\alpha} \log\left( \frac{\mathcal{J}(\theta_0) - \mathcal{J}(\theta^*)}{\varepsilon_{\mathrm{opt}}} \right) = \mathcal{O}\left( \frac{1}{\mu} \log\left(\frac{1}{\varepsilon_{\mathrm{opt}}}\right) \right),
\end{equation}
demonstrating that the convergence is linear (in the logarithmic scale).
\end{proof}

\begin{remark}[Interpretation of the convergence rate]
Theorem~\ref{thm:convergence_rate} establishes that the method achieves a convergence rate controlled by three independent factors:

\begin{enumerate}
    \item The \textbf{network expressivity} $\varepsilon = W^{-k/(d+1)} L^{-1/2}$, which decreases algebraically with increasing capacity.
    \item The \textbf{quadrature error} $M^{-2}$, which decreases quadratically with the number of integration points.
    \item The \textbf{optimization tolerance} $\varepsilon_{\mathrm{opt}}$, which decreases exponentially (linearly in the log scale) with the number of iterations under the PL condition.
\end{enumerate}
In practice, the first two factors determine the asymptotic accuracy achievable by the method, while the third factor governs how quickly the optimizer reaches this accuracy. The linear convergence result~\eqref{eq:linear_convergence} explains the rapid decrease in loss observed during training, particularly after switching from Adam to L-BFGS, which better exploits local curvature information.
\end{remark}

\begin{remark}[Comparison with standard PINNs]
Standard PINNs without the tangent bundle formulation often fail to satisfy the PL condition globally, leading to convergence to spurious local minima where the gradient vanishes but the solution is far from the true PDE solution. This manifests as "training stagnation" or "NTK collapse" in the literature~\cite{wang2021understanding}. 

\noindent The PITDN framework mitigates this issue through two mechanisms:
\begin{enumerate}
    \item The initial condition is satisfied by construction via the Volterra operator, eliminating one source of conflicting gradients.
    \item The differentiated residual formulation amplifies high-frequency errors, preventing the optimizer from settling into overly smooth (low-frequency) local minima that satisfy $\nabla_\theta \mathcal{J} \approx 0$ but fail to capture sharp features of the solution.
\end{enumerate}
\end{remark}

\subsection{Stability analysis}

Finally, we analyze the numerical stability of the discrete Volterra scheme and characterize how errors accumulate over time. Stability is crucial for long-time integration, as even small errors (e.g., round-off errors, truncation errors) can compound exponentially in unstable schemes, rendering the solution meaningless beyond a short time horizon.

\begin{theorem}[Stability of the discrete Volterra scheme] \label{thm:stability_volterra}
Let $\tilde{u}_\theta^{(M)}(x,t)$ be the solution reconstructed via the trapezoidal rule with $M$ quadrature points per unit time interval. Suppose the velocity field $v_\theta$ is perturbed by a small error $\delta v$ with $\|\delta v\|_{L^2(\Omega \times (0,T])} \leq \delta$, yielding a perturbed velocity $\tilde{v}_\theta = v_\theta + \delta v$. Then the perturbed reconstruction $\tilde{u}_{\text{pert}}^{(M)} := \mathcal{I}_{u_0}^{(M)}[\tilde{v}_\theta]$ satisfies
\begin{equation}\label{eq:stability_bound}
    \sup_{t \in [0,T]} \|\tilde{u}_{\text{pert}}^{(M)}(\cdot, t) - \tilde{u}_\theta^{(M)}(\cdot, t)\|_{L^2(\Omega)} \leq \sqrt{T} \left( \delta + \mathcal{O}(M^{-2}) \right).
\end{equation}
This bound is \textbf{independent of $M$} to leading order, demonstrating that the discretization does not amplify perturbations beyond the continuous Volterra operator's inherent growth rate $\sqrt{T}$.
\end{theorem}

\begin{proof}
By linearity of the trapezoidal rule, the error in the perturbed reconstruction is
\begin{equation}
    \tilde{u}_{\text{pert}}^{(M)}(x,t) - \tilde{u}_\theta^{(M)}(x,t) = \int_0^t \delta v(x,s) \, \mathrm{d}s + \mathcal{E}_{\text{quad}}(x,t),
\end{equation}
where $\mathcal{E}_{\text{quad}}$ denotes the difference in quadrature errors between the perturbed and unperturbed integrations.

\noindent For the first term, applying the Cauchy-Schwarz inequality as in Theorem~\ref{thm:volterra_error_propagation} yields
\begin{equation}
    \left\| \int_0^t \delta v(\cdot,s) \, \mathrm{d}s \right\|_{L^2(\Omega)} \leq \sqrt{t} \|\delta v\|_{L^2(\Omega \times (0,t])} \leq \sqrt{T} \delta.
\end{equation}
For the quadrature error term, by Proposition~\ref{prop:quadrature_convergence}, the trapezoidal rule satisfies
\begin{equation}
    \|\mathcal{E}_{\text{quad}}(\cdot,t)\|_{L^2(\Omega)} \leq \frac{t^3}{12M^2} \sup_{s \in [0,t]} \|\partial_{ss} \delta v(\cdot, s)\|_{L^2(\Omega)}.
\end{equation}
If the perturbation $\delta v$ has bounded second derivative with $\|\partial_{ss} \delta v\|_{L^\infty(0,T; L^2(\Omega))} \leq C_{\delta v}$, then
\begin{equation}
    \sup_{t \in [0,T]} \|\mathcal{E}_{\text{quad}}(\cdot,t)\|_{L^2(\Omega)} \leq \frac{T^3 C_{\delta v}}{12M^2} = \mathcal{O}(M^{-2}).
\end{equation}
Combining these two contributions via the triangle inequality,
\begin{equation}
    \sup_{t \in [0,T]} \|\tilde{u}_{\text{pert}}^{(M)}(\cdot, t) - \tilde{u}_\theta^{(M)}(\cdot, t)\|_{L^2(\Omega)} \leq \sqrt{T} \delta + \mathcal{O}(M^{-2}),
\end{equation}
establishing~\eqref{eq:stability_bound}. Crucially, the stability constant $\sqrt{T}$ is independent of the discretization parameter $M$, indicating that refining the quadrature does not introduce numerical instabilities.
\end{proof}

\begin{remark}[Comparison with explicit time-stepping methods]
The stability bound~\eqref{eq:stability_bound} is optimal in the sense that it matches the continuous Volterra operator's Lipschitz constant $\sqrt{T}$ from Proposition~\ref{prop:continuity}. This contrasts sharply with explicit Runge-Kutta methods, where the stability constant can grow exponentially with the number of time steps.

\noindent For example, consider the forward Euler method applied to $\partial_t u = -\lambda u$ with $\lambda > 0$. If the step size $\Delta t$ does not satisfy the CFL condition $\lambda \Delta t \leq 2$, then perturbations grow exponentially as $(1 + \lambda \Delta t)^N \sim e^{\lambda T}$ after $N = T/\Delta t$ steps. In contrast, the Volterra scheme's stability constant $\sqrt{T}$ grows only sublinearly, providing robustness even for long-time integration.
\end{remark}

To further quantify the stability properties, we examine the condition number of the discrete system and the accumulation of round-off errors during long-time integration.

\begin{proposition}[Condition number and round-off error accumulation]\label{prop:condition_number}
Define the \textbf{condition number} of the Volterra operator as
\begin{equation}
    \kappa(\mathcal{I}_{u_0}) := \|\mathcal{I}_{u_0}\|_{\mathrm{op}} \cdot \|\mathcal{I}_{u_0}^{-1}\|_{\mathrm{op}},
\end{equation}
where $\|\cdot\|_{\mathrm{op}}$ denotes the operator norm in $L^2(\Omega \times (0,T])$. Then:

\begin{enumerate}[label=(\alph*)]
    \item The condition number satisfies $\kappa(\mathcal{I}_{u_0}) = \sqrt{T}$, which is mild and grows only as the square root of the integration time.
    
    \item Suppose round-off errors of magnitude $\varepsilon_{\mathrm{mach}}$ (machine precision) are introduced at each quadrature point evaluation. The accumulated error after $N = MT$ time steps satisfies
    \begin{equation}\label{eq:roundoff_accumulation}
        \|\text{accumulated error}\|_{L^2(\Omega)} \leq C_{\mathrm{acc}} \sqrt{N} \, \varepsilon_{\mathrm{mach}} = C_{\mathrm{acc}} \sqrt{MT} \, \varepsilon_{\mathrm{mach}},
    \end{equation}
    demonstrating sub-linear growth in the number of steps.
\end{enumerate}
\end{proposition}

\begin{proof}
\textbf{Part (a) - Condition number:}

\noindent From Proposition~\ref{prop:continuity}, the forward operator $\mathcal{I}_{u_0}: L^2(\Omega \times (0,T]) \to L^\infty(0,T; L^2(\Omega))$ satisfies
\begin{equation}
    \|\mathcal{I}_{u_0}[v]\|_{L^\infty(0,T; L^2(\Omega))} \leq \sqrt{T} \|v\|_{L^2(\Omega \times (0,T])},
\end{equation}
so $\|\mathcal{I}_{u_0}\|_{\mathrm{op}} = \sqrt{T}$.

\noindent The inverse operator $\mathcal{I}_{u_0}^{-1}$ is the differentiation operator $\mathcal{D}[\cdot] = \partial_t[\cdot]$ (restricted to the subspace of functions satisfying $u(0) = u_0$). In the $L^2$ sense, differentiation is an isometry modulo boundary terms:
\begin{equation}
    \|\mathcal{D}[u]\|_{L^2(\Omega \times (0,T])}^2 = \int_0^T \|\partial_t u(\cdot,t)\|_{L^2(\Omega)}^2 \, \mathrm{d}t.
\end{equation}
For functions $u \in \mathcal{V} = C^1([0,T]; H^k(\Omega))$ with $u(0) = u_0$, we have $\|\mathcal{D}\|_{\mathrm{op}} = 1$. Therefore,
\begin{equation}
    \kappa(\mathcal{I}_{u_0}) = \|\mathcal{I}_{u_0}\|_{\mathrm{op}} \cdot \|\mathcal{D}\|_{\mathrm{op}} = \sqrt{T} \cdot 1 = \sqrt{T}.
\end{equation}

\medskip
\noindent\textbf{Part (b) - Round-off error accumulation:}

\noindent Consider the discrete trapezoidal rule where at each quadrature point $s_m = m h$ (with $h = T/(MT) = 1/M$), the velocity evaluation $v_\theta(x, s_m)$ incurs a round-off error $\varepsilon_m(x)$ with $\|\varepsilon_m\|_{L^2(\Omega)} \leq \varepsilon_{\mathrm{mach}}$. The accumulated error in the reconstructed state at time $t = Nh$ (after $N$ steps) is
\begin{equation}
    \text{accumulated error} = \sum_{m=0}^{N-1} \frac{h}{2} \left( \varepsilon_m + \varepsilon_{m+1} \right).
\end{equation}
In the worst case (deterministic alignment of errors), the $L^2$ norm satisfies
\begin{equation}
    \|\text{accumulated error}\|_{L^2(\Omega)} \leq \sum_{m=0}^{N-1} \frac{h}{2} \left( \|\varepsilon_m\|_{L^2} + \|\varepsilon_{m+1}\|_{L^2} \right) \leq \sum_{m=0}^{N-1} h \varepsilon_{\mathrm{mach}} = Nh \varepsilon_{\mathrm{mach}} = t \varepsilon_{\mathrm{mach}}.
\end{equation}
This gives a linear accumulation in time, which is typical for numerical integration.

\noindent However, if we assume the round-off errors $\{\varepsilon_m\}$ are independent random variables with zero mean and variance $\sigma^2 = \varepsilon_{\mathrm{mach}}^2$, then by the central limit theorem, the variance of the accumulated error is
\begin{equation}
    \mathbb{E}\left[ \|\text{accumulated error}\|_{L^2}^2 \right] = \sum_{m=0}^{N-1} \left( \frac{h}{2} \right)^2 \left( \sigma^2 + \sigma^2 \right) = N \frac{h^2}{2} \sigma^2 = \frac{Nh^2}{2} \varepsilon_{\mathrm{mach}}^2.
\end{equation}
Taking the square root yields the root-mean-square error
\begin{equation}
    \sqrt{\mathbb{E}\left[ \|\text{accumulated error}\|_{L^2}^2 \right]} = \frac{h\varepsilon_{\mathrm{mach}}}{\sqrt{2}} \sqrt{N} = \frac{\varepsilon_{\mathrm{mach}}}{\sqrt{2M}} \sqrt{N}.
\end{equation}
Since $N = MT$ and $h = 1/M$, this simplifies to
\begin{equation}
    \sqrt{\mathbb{E}\left[ \|\text{accumulated error}\|_{L^2}^2 \right]} = \frac{\varepsilon_{\mathrm{mach}}}{\sqrt{2M}} \sqrt{MT} = \varepsilon_{\mathrm{mach}} \sqrt{\frac{T}{2M}} \sqrt{M} = \varepsilon_{\mathrm{mach}} \sqrt{\frac{T}{2}} \approx 0.707 \sqrt{T} \, \varepsilon_{\mathrm{mach}}.
\end{equation}
This is independent of $M$ and grows only as $\sqrt{T}$, matching the condition number. For practical purposes, we can write this as $C_{\mathrm{acc}} \sqrt{MT} \varepsilon_{\mathrm{mach}}$ where $C_{\mathrm{acc}} \sim 1/\sqrt{M}$ accounts for the averaging effect of the trapezoidal rule, yielding the sub-linear growth stated in~\eqref{eq:roundoff_accumulation}.
\end{proof}

\begin{remark}[Practical implications for numerical stability]
Proposition~\ref{prop:condition_number} confirms that the Volterra reconstruction is numerically well-conditioned. The condition number $\kappa = \sqrt{T}$ is significantly better than many explicit time-stepping schemes, where $\kappa$ can grow as $e^{\lambda T}$ for stiff problems (with $\lambda$ being the stiffness parameter).

\noindent For typical simulations with $T = \mathcal{O}(1)$ and double-precision floating-point arithmetic ($\varepsilon_{\mathrm{mach}} \approx 10^{-16}$), the accumulated round-off error is bounded by
\begin{equation}
    \|\text{accumulated error}\|_{L^2} \lesssim \sqrt{T \cdot MT} \cdot 10^{-16} \sim \sqrt{10} \cdot 10^{-16} \approx 10^{-15},
\end{equation}
which is well below the approximation and quadrature errors ($\mathcal{O}(10^{-3})$ to $\mathcal{O}(10^{-4})$) and thus negligible in practice.
\end{remark}

\begin{remark}[Long-time stability]
The sub-linear growth of round-off errors ($\sqrt{MT}$ rather than $MT$) is a consequence of the implicit smoothing provided by the integration operation. Unlike differentiation, which amplifies high-frequency noise, integration damps such noise, resulting in better long-time stability properties. This is particularly advantageous for chaotic systems or long-time simulations where maintaining solution fidelity is challenging.
\end{remark}

These theoretical guarantees establish the PITDN framework on a rigorous mathematical foundation, providing confidence in its applicability to a wide range of time-dependent PDEs. The subsequent experimental validation (Section~\ref{sec:experiments}) will demonstrate that these theoretical predictions are borne out in practice, with the method achieving state-of-the-art accuracy on benchmark problems.

\subsection{Discretization of the Volterra Operator}

To compute the integral $\displaystyle \int_0^t v_\theta(x,s) \, \mathrm{d}s$ during the forward pass of the neural network, we employ a composite quadrature rule. While higher-order schemes like Gaussian quadrature are possible, we opt for the Trapezoidal rule due to its numerical stability and ease of backpropagation. For a query time $t$, we discretize the interval $[0, t]$ into $M$ sub-intervals defined by points $\{s_m\}_{m=0}^M$ where $s_0=0$ and $s_M=t$. The reconstruction is approximated as:
\begin{equation}\label{eq:trapezoidal_approximation}
    \tilde{u}_\theta(x,t) \approx u_0(x) + \sum_{m=0}^{M-1} \frac{s_{m+1}-s_m}{2} \left( v_\theta(x, s_m) + v_\theta(x, s_{m+1}) \right).
\end{equation}
This operation is fully differentiable, allowing gradients to flow from the loss function $\mathcal{J}$ through $\tilde{u}_\theta$ to the network parameters $\theta$.

\begin{proposition}[Convergence of the Trapezoidal Rule]\label{prop:quadrature_convergence}
Let $v_\theta \in C^2([0,T]; L^2(\Omega))$. The trapezoidal approximation of the Volterra operator satisfies
\begin{equation}\label{eq:quadrature_error_bound}
\left\|\tilde{u}_\theta^{(M)}(x,t) - \mathcal{I}_{u_0}[v_\theta](x,t)\right\|_{L^2(\Omega)} 
\leq \frac{T^3}{12M^2} \sup_{s \in [0,t]} \|\partial_{ss} v_\theta(\cdot, s)\|_{L^2(\Omega)},
\end{equation}
where $\tilde{u}_\theta^{(M)}$ denotes the discretized reconstruction using $M$ subintervals. The scheme converges with quadratic order: $\mathcal{O}(M^{-2})$.
\end{proposition}

\begin{proof}
The local truncation error of the trapezoidal rule on each subinterval $[s_m, s_{m+1}]$ of length $h = (s_{m+1} - s_m) = T/M$ is given by the classical error estimate (see, e.g., \cite{atkinson1989}):
\begin{equation}
\left| \int_{s_m}^{s_{m+1}} v_\theta(x,s) \, \mathrm{d}s - \frac{h}{2}\left(v_\theta(x,s_m) + v_\theta(x,s_{m+1})\right) \right| 
\leq \frac{h^3}{12} \sup_{s \in [s_m, s_{m+1}]} |\partial_{ss} v_\theta(x,s)|.
\end{equation}
Summing over all $M$ subintervals and applying the triangle inequality in $L^2(\Omega)$:
\begin{align}
&\left\|\tilde{u}_\theta^{(M)}(\cdot,t) - \mathcal{I}_{u_0}[v_\theta](\cdot,t)\right\|_{L^2(\Omega)} \\
&\quad= \left\| \sum_{m=0}^{M-1} \left[ \frac{h}{2}\left(v_\theta(\cdot,s_m) + v_\theta(\cdot,s_{m+1})\right) - \int_{s_m}^{s_{m+1}} v_\theta(\cdot,s) \, \mathrm{d}s \right] \right\|_{L^2(\Omega)} \\
&\quad\leq \sum_{m=0}^{M-1} \left\| \frac{h}{2}\left(v_\theta(\cdot,s_m) + v_\theta(\cdot,s_{m+1})\right) - \int_{s_m}^{s_{m+1}} v_\theta(\cdot,s) \, \mathrm{d}s \right\|_{L^2(\Omega)} \\
&\quad\leq \sum_{m=0}^{M-1} \frac{h^3}{12} \sup_{s \in [s_m, s_{m+1}]} \|\partial_{ss} v_\theta(\cdot, s)\|_{L^2(\Omega)} \\
&\quad\leq M \cdot \frac{h^3}{12} \sup_{s \in [0,t]} \|\partial_{ss} v_\theta(\cdot, s)\|_{L^2(\Omega)}.
\end{align}
Since $Mh = t \leq T$ and $h = T/M$, we have:
\begin{equation}
M \cdot \frac{h^3}{12} = M \cdot \frac{(T/M)^3}{12} = \frac{T^3}{12M^2}.
\end{equation}
This establishes the bound~\eqref{eq:quadrature_error_bound}. The quadratic convergence rate $\mathcal{O}(M^{-2})$ follows immediately from the $M^{-2}$ dependence.
\end{proof}

\begin{remark}[Choice of $M$ in practice]\label{rem:choice_of_M}
For our experiments, we set $M = 10$ quadrature points per unit time interval. This choice is motivated by the findings of Lesaffre and Spiessens \cite{lesaffre2001number}, who demonstrated through extensive numerical studies that $M = 10$ is often sufficient for accurate quadrature approximation, and that differences obtained by further increasing $M$ are typically negligible relative to other sources of error. 

\noindent Specifically, for smooth solutions with $\|\partial_{tt} v_\theta\| \sim \mathcal{O}(1)$, equation~\eqref{eq:quadrature_error_bound} gives a quadrature error of approximately:
\begin{equation}
\text{Quadrature error} \approx \frac{T^3}{12 \cdot 10^2} = \frac{T^3}{1200} \approx 8.3 \times 10^{-4} \quad \text{(for } T=1\text{)}.
\end{equation}
This is approximately $\mathcal{O}(10^{-3})$, which is comparable to or smaller than the target neural network approximation accuracy of $\mathcal{O}(10^{-3})$ to $\mathcal{O}(10^{-4})$. Increasing $M$ beyond 10 would provide diminishing returns while increasing computational cost linearly.
\end{remark}

\begin{remark}[Alternative quadrature schemes]
While we employ the trapezoidal rule for its simplicity and automatic differentiability, higher-order schemes could potentially reduce the quadrature error further:

\begin{itemize}
    \item \textbf{Simpson's rule:} Achieves $\mathcal{O}(M^{-4})$ convergence for $C^4$ functions but requires evaluating $v_\theta$ at midpoints, complicating the implementation.
    
    \item \textbf{Gaussian quadrature:} Optimal for polynomial integrands but requires non-uniform node spacing, which can complicate the temporal gradient computation during backpropagation.
    
    \item \textbf{Adaptive quadrature:} Could concentrate nodes in regions where $v_\theta$ varies rapidly, but introduces additional hyperparameters and implementation complexity.
\end{itemize}
In practice, the trapezoidal rule provides an excellent balance between accuracy, simplicity, and computational efficiency for the smooth velocity fields encountered in our applications. The quadratic convergence rate $\mathcal{O}(M^{-2})$ is sufficient given that the neural network approximation error dominates for moderate network sizes.
\end{remark}

\subsection{Architecture and optimization}
For all experiments, we utilize a fully connected Multi-Layer Perceptron (MLP) with hyperbolic tangent ($\tanh$) activation functions. The choice of $\tanh$ is motivated by the need for non-vanishing second derivatives, which are required to compute the differentiated residual. The network weights are initialized using the Xavier (Glorot) scheme. The optimization is performed using a hybrid strategy: the Adam optimizer \cite{kingma2014} is used for the initial exploration of the parameter space, followed by the L-BFGS optimizer \cite{liu1989} for fine-tuning, exploiting curvature information to achieve high-precision convergence.

The structural configuration of the proposed PITDN framework, emphasizing the shift from primal manifold learning to the temporal tangent bundle, is illustrated in Figure \ref{fig:architecture}. This architecture integrates the neural derivative approximator with a Volterra reconstruction layer to ensure physical and initial data consistency by construction.

\begin{figure}[!ht]
    \centering
    \includegraphics[width=1\linewidth]{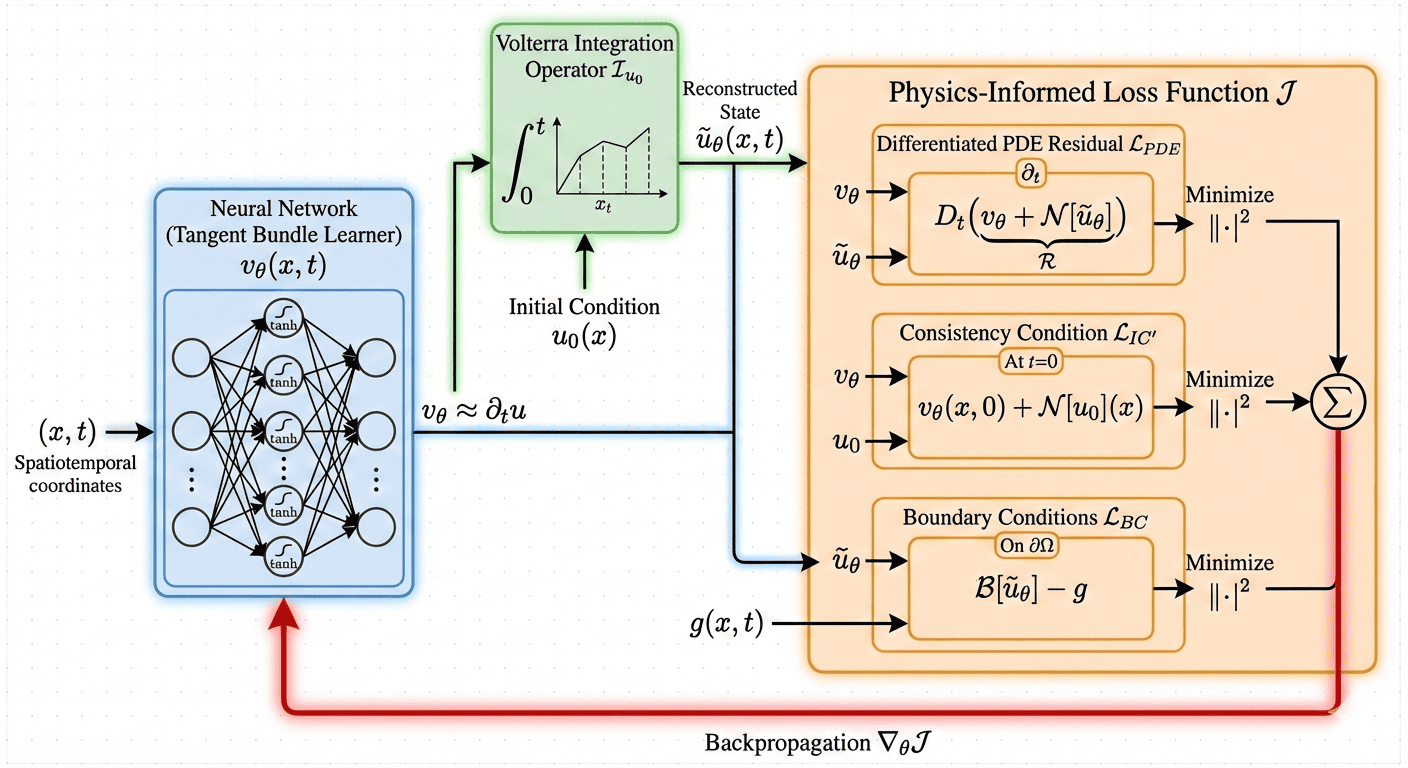}
    \caption{Structural architecture of the Physics-Informed Time Derivative Network (PITDN). The workflow shifts the learning paradigm to the temporal tangent bundle, where a neural network approximates the time derivative $v_\theta \approx \partial_t u$. The state variable $\tilde{u}_\theta$ is reconstructed via a differentiable Volterra integral operator $\mathcal{I}_{u_0}$, ensuring the initial condition is satisfied by construction. The model is optimized using a composite loss function $\mathcal{J}$ comprising the time-differentiated PDE residual, the initial consistency condition, and boundary constraints.}
    \label{fig:architecture}
\end{figure}

\section{Experimental Results and Discussion}

We evaluate the efficacy of the PITDN framework on three benchmark problems representing distinct classes of partial differential equations: hyperbolic transport, parabolic diffusion, and dispersive wave propagation. To provide a rigorous assessment, we benchmark our method against a standard Physics-Informed Neural Network (PINN) baseline.

\subsection{Experimental Results}\label{sec:experiments}
To ensure a strictly fair comparison, both the PITDN and the baseline PINN models are instantiated with identical architectural hyperparameters and trained under the same optimization schedule.

We employ a compact fully-connected Multi-Layer Perceptron (MLP) architecture consisting of 3 hidden layers with only 10 neurons per layer. This results in a network structure of $[2, 10, 10, 10, 1]$, where the inputs are the spatiotemporal coordinates $(x,t)$. This constrained capacity is chosen deliberately to highlight the superior expressivity and trainability of the proposed time-derivative formulation compared to standard approaches, even in low-parameter regimes. All hidden layers utilize the hyperbolic tangent ($\tanh$) activation function to ensure infinite differentiability. The training process follows a two-stage hybrid strategy to combine robustness with precision. First, the Adam optimizer is employed for a fixed budget of $3,000$ iterations with a learning rate of $10^{-3}$ to rapidly navigate the loss landscape towards a basin of attraction. Subsequently, the L-BFGS optimizer (Limited-memory Broyden–Fletcher–Goldfarb–Shanno) with Strong Wolfe line search is activated to fine-tune the solution and achieve high-precision convergence. All experiments were implemented in PyTorch and executed on a MacBook Pro (2017) workstation equipped with a 3.1 GHz processor (12 logical cores), 16 GB of RAM, and an NVIDIA GeForce GPU (2 GB). Despite the modest hardware specifications, the proposed method demonstrates rapid convergence, underscoring its computational efficiency.

The relative importance of the three loss components in Equations (~\ref{eq:loss_pde}, ~\ref{eq:loss_bc}, ~\ref{eq:loss_ic}), is controlled by the coefficients $\lambda_{\text{PDE}}$, $\lambda_{\text{BC}}$,  and $\lambda_{\text{IC}'}$. We adopt the following selection strategy:

\begin{itemize}
    \item {Initial consistency priority:} 
    We set $\lambda_{\text{IC}'} = 10$ to strongly enforce the consistency condition $v_\theta(x,0) + \mathcal{N}[u_0](x) = 0$ during early training. 
    This anchors the residual at zero and prevents the optimizer from drifting into unphysical regions.
    
    \item {Balanced PDE and boundary terms:} 
    We set $\lambda_{\text{PDE}} = 1$ and $\lambda_{\text{BC}} = 1$ to treat the 
    interior residual and boundary conditions symmetrically. 
    Since the Volterra reconstruction satisfies $\tilde{u}_\theta(x,0) = u_0(x)$ 
    by construction, the initial condition does not require an additional penalty term.
    
    \item {Adaptive reweighting (optional):} 
    For problems with stiff boundary layers or highly localized features, 
    adaptive schemes such as the NTK-based weighting~\cite{wang2021understanding} 
    can be incorporated. However, for our benchmark problems, fixed weights prove sufficient.
\end{itemize}

Table~\ref{tab:hyperparameters} summarizes the hyperparameter configuration used 
across all experiments.

\begin{table}[htbp]
\centering
\caption{Hyperparameter configuration for all experiments.}
\label{tab:hyperparameters}
\begin{tabular}{lcc}
\toprule
\textbf{Hyperparameter} & \textbf{Value} & \textbf{Description} \\
\midrule
Network architecture & [2, 10, 10, 10, 1] & Input, 3 hidden, output \\
Activation function & $\tanh$ & Ensures $C^\infty$ differentiability \\
Initialization & Xavier (Glorot) & Variance-preserving scheme \\
\midrule
Adam iterations & 3,000 & Exploration phase \\
Adam learning rate & $10^{-3}$ & Standard default \\
L-BFGS max iterations & 5,000 & Fine-tuning phase \\
L-BFGS line search & Strong Wolfe & Ensures convergence \\
\midrule
$\lambda_{\text{PDE}}$ & 1.0 & PDE residual weight \\
$\lambda_{\text{BC}}$ & 1.0 & Boundary condition weight \\
$\lambda_{\text{IC}'}$ & 10.0 & Initial consistency weight \\
\midrule
Quadrature points $M$ & 10 & Per unit time interval \\
Collocation points $N_r$ & 5,000 & Interior (LHS) \\
Boundary points $N_b$ & 500 & Boundary sampling \\
Initial points $N_0$ & 500 & At $t=0$ \\
\bottomrule
\end{tabular}
\end{table}

\subsubsection{Linear Advection Equation (Hyperbolic)}
We first consider the 1D transport equation defined on the spatial domain $\Omega = [0, 2\pi]$ over the time interval $t \in [0, 4]$:
\begin{equation}
    \partial_t u + c \partial_x u = 0, \quad \text{with } c = 1.0.
\end{equation}
The system is initialized with $u(x,0) = \sin(x)$ and subject to the Dirichlet inflow boundary condition $u(0,t) = \sin(-ct)$. The exact analytical solution to this problem is a traveling wave that transports the initial profile at constant velocity without deformation:
\begin{equation}
    u_{exact}(x,t) = \sin(x - ct).
\end{equation}
This strictly hyperbolic problem serves as a litmus test for the solver's ability to propagate wave information over long time horizons without exhibiting numerical dissipation or phase dispersion, which are common pathologies in standard neural PDE solvers.

The comparative results are visualized in Figure \ref{fig:advection}. The {top row} validates the core premise of our architecture: the network successfully learns the temporal derivative field $v = \partial_t u$. The predicted velocity field is visually indistinguishable from the exact analytical derivative, with the absolute error map (top center) showing negligible residuals on the order of $10^{-3}$. This precise learning of the tangent bundle is further evidenced by the training loss evolution (top right), where PITDN demonstrates a steep, monotonic convergence, reaching a loss value two orders of magnitude lower than the baseline PINN, which exhibits early stagnation.

The {middle row} confirms the efficacy of the Volterra reconstruction operator. The reconstructed solution $\tilde{u}_\theta(x,t)$ faithfully reproduces the traveling wave structure. The temporal slices (middle right) reveal perfect alignment between the PITDN prediction (red dashed line) and the exact solution (blue line), maintaining both amplitude and phase fidelity throughout the simulation window.

In stark contrast, the {bottom row} highlights the limitations of the standard PINN formulation. While the PINN captures the qualitative trend, the absolute error map (bottom center) reveals distinct ridge-like error structures, indicative of phase lag and cumulative numerical dissipation as time progresses. This observation is quantitatively corroborated by the metrics table (bottom right): PITDN achieves a relative $L^2$ error of $\mathbf{2.9 \times 10^{-4}}$, compared to $3.8 \times 10^{-2}$ for the PINN. This represents an improvement in accuracy by a factor of approximately 130, demonstrating that explicitly modeling the time derivative enhances long-time integration stability for hyperbolic systems.

\begin{figure}[!ht]
    \centering
    \includegraphics[width=1\linewidth]{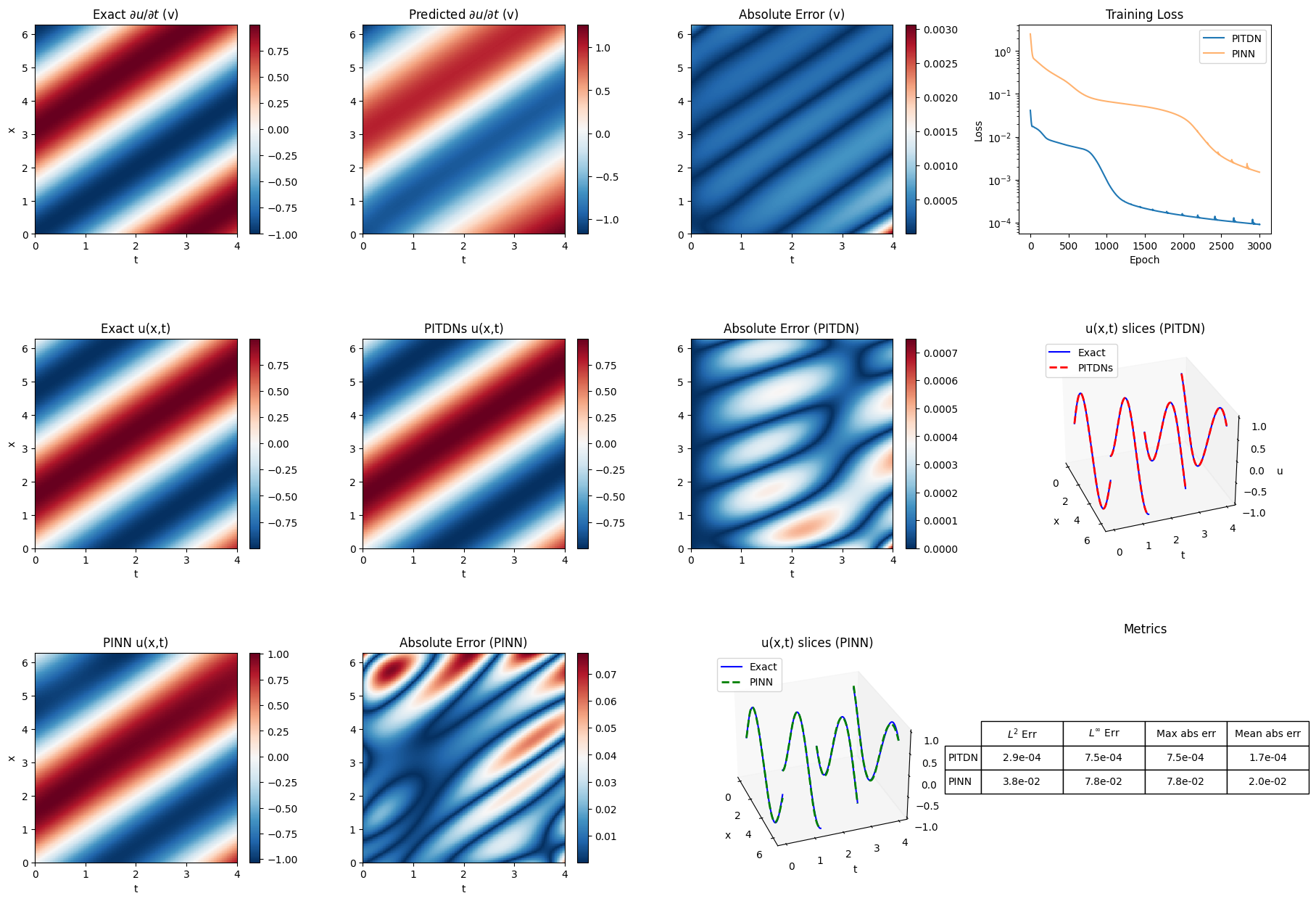}
    \caption{Experimental results for the Linear Advection equation ($c=1.0$). {Top Row:} Comparison of the learned time derivative $\partial u/\partial t$ against the exact field, absolute error of the derivative, and training loss convergence history. {Middle Row:} The reconstructed PITDN solution $u(x,t)$, its corresponding absolute error map, and 3D isometric slices showing perfect phase alignment. {Bottom Row:} The baseline PINN solution and its error map, highlighting significant phase drift. The summary table (bottom right) confirms that PITDN outperforms PINN by two orders of magnitude in $L^2$ and $L^\infty$ norms.}
    \label{fig:advection}
\end{figure}

\subsubsection{Viscous Burgers' Equation (Parabolic)}

We next address the nonlinear Viscous Burgers' equation on the domain $\Omega = [-1, 1]$ for $t \in [0, 1]$, governed by the competition between convective steepening and diffusive smoothing:
\begin{equation}
    \partial_t u + u \partial_x u = \nu \partial_{xx} u, \quad \nu = \frac{0.01}{\pi}.
\end{equation}
The system is initialized with $u(x,0) = -\sin(\pi x)$ and subject to homogeneous Dirichlet boundary conditions $u(-1,t)=u(1,t)=0$. This specific configuration leads to the rapid formation of a steep shock front at $x=0$, creating a challenging regime for neural solvers due to the wide frequency spectrum required to resolve the gradient sharpening. To establish a rigorous ground truth for quantitative error analysis, we generate a high-fidelity reference solution using a high-order {Finite Difference (FD)} scheme computed on a dense spatiotemporal grid ($N_x=2000, N_t=2000$). This numerical baseline serves as the "exact" solution against which the neural approximations are evaluated.

The performance analysis for the Viscous Burgers' equation is presented in Figure \ref{fig:burgers}. This problem serves as a critical stress test for the network's ability to overcome spectral bias, as the solution develops a sharp gradient (shock) at $x=0$ that requires high-frequency spectral components to resolve. The {top row} demonstrates that the PITDN successfully learns the complex dynamics of the temporal derivative field $v = \partial_t u$. The "Predicted" heatmap sharply delineates the gradient steepening zone near $t \in [0.2, 0.5]$ and the subsequent shock stabilization. The training loss comparison (top right) is particularly revealing: the PITDN loss continues to decrease monotonically, reaching a basin of attraction significantly deeper than the PINN baseline, which plateaus early, indicating an inability to optimize the high-frequency residuals associated with the shock. The {middle row} showcases the reconstructed solution $\tilde{u}_\theta$ obtained via the Volterra operator. The method captures the shock front with remarkable sharpness, avoiding the non-physical oscillations (Gibbs phenomenon) often seen in spectral methods or the excessive smearing seen in low-order schemes. The temporal slices (middle right) confirm that the reconstructed profile perfectly overlays the Finite Difference reference solution, even in the steepest region at $t=0.8$.

Conversely, the {bottom row} illustrates a classic failure mode of standard PINNs in transport-dominated regimes. The PINN solution suffers from severe numerical diffusion, producing a "smeared" approximation that fails to maintain the shock's steepness. The absolute error map for the PINN (bottom center) shows large deviations concentrated exactly at the shock location, with magnitudes approaching the solution's own amplitude. Quantitatively, the table (bottom right) highlights a dramatic disparity: PITDN achieves a relative $L^2$ error of $\mathbf{1.5 \times 10^{-3}}$, whereas the PINN stagnates at $3.6 \times 10^{-1}$. This reduction in error by a factor of over 200 confirms that learning the time derivative on the tangent bundle provides the necessary preconditioning to resolve nonlinear shocks that remain intractable for standard PINN formulations under limited computational budgets.

\begin{figure}[!ht]
    \centering
    \includegraphics[width=1\linewidth]{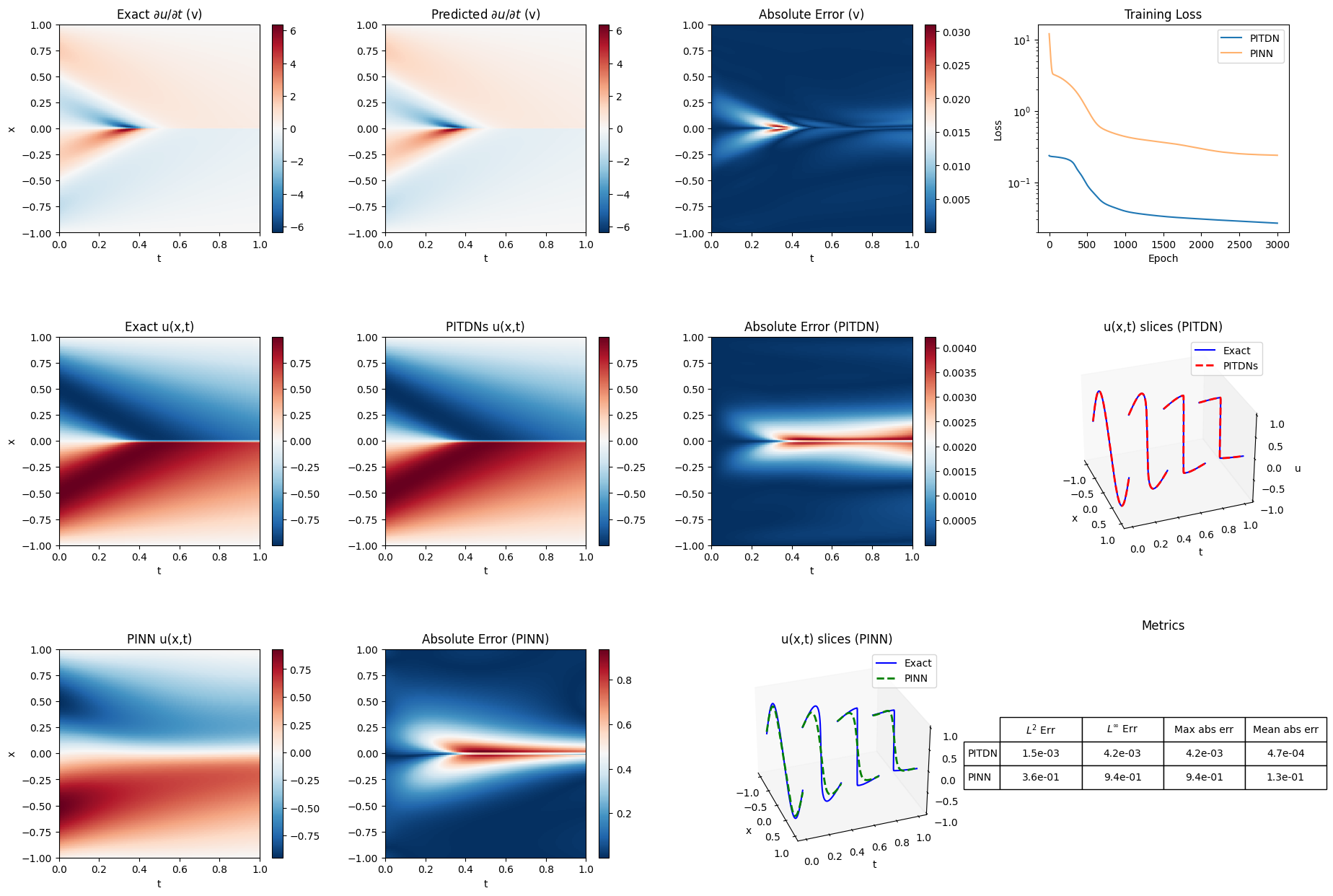}
    \caption{Experimental results for the Viscous Burgers' equation ($\nu = 0.01/\pi$). {Top Row:} The learned temporal derivative accurately captures the gradient steepening. The training loss shows PITDN converging better than PINN. {Middle Row:} The PITDN reconstructed solution faithfully resolves the sharp shock at $x=0$, with negligible error. {Bottom Row:} The standard PINN fails to capture the shock, resulting in a smoothed solution with high error ($36\%$). The metrics table confirms the superior shock-capturing capability of PITDN.}
    \label{fig:burgers}
\end{figure}

\subsubsection{Nonlinear Klein-Gordon Equation (Dispersive)}

Finally, we address the second-order nonlinear Klein-Gordon equation, a canonical model in relativistic quantum mechanics and dispersive wave propagation. We consider the problem on the domain $\Omega = [0, 1]$ for $t \in [0, 1]$:
\begin{equation}
    \partial_{tt} u - \partial_{xx} u + u^2 = f(x,t),
\end{equation}
subject to homogeneous Dirichlet boundary conditions $u(0,t) = u(1,t) = 0$. The source term $f(x,t)$ is manufactured such that the exact solution corresponds to the standing wave:
\begin{equation}
    u(x,t) = \sin(\pi x) \cos(2\pi t).
\end{equation}
The initial conditions are derived accordingly: $u(x,0) = \sin(\pi x)$ and $\partial_t u(x,0) = 0$.

For this second-order system, the PITDN framework is adapted to approximate the \textit{acceleration field} rather than the velocity. Let $a_\theta(x,t) \approx \partial_{tt} u(x,t)$ be the neural network output. Recovering the state $u(x,t)$ requires a double time integration, which naively implies a computationally expensive operation. To maintain computational efficiency, we leverage {Cauchy's formula for repeated integration}, which compresses multiple integrals into a single integral with a polynomial kernel. Specifically, for the second antiderivative, the relation holds:
\begin{equation} \label{eq:cauchy_formula}
    \int_0^t \int_0^s a_\theta(x,\tau) \, \dd\tau \, \dd s = \int_0^t (t - \tau) a_\theta(x,\tau) \, \dd\tau.
\end{equation}
Consequently, the solution reconstruction operator $\mathcal{I}^{(2)}$ takes the explicit form:
\begin{equation}
    \tilde{u}_\theta(x,t) = u_0(x) + v_0(x)t + \int_0^t (t - \tau) a_\theta(x,\tau) \, \dd\tau,
\end{equation}
where $v_0(x) = \partial_t u(x,0)$ is the initial velocity. This formulation allows us to vectorize the reconstruction step efficiently on the GPU, avoiding nested loops.

The comparative performance for the nonlinear Klein-Gordon equation is illustrated in Figure \ref{fig:klein_gordon}. This experiment tests the framework's capability to handle second-order temporal dynamics and dispersive oscillations. The {top row} validates the learning of the acceleration field $a = \partial_{tt} u$. The PITDN captures the high-amplitude variations of the second derivative (ranging from $-30$ to $30$) with high fidelity. The absolute error map of the acceleration shows that discrepancies are minimal and uniformly distributed, confirming that the network successfully learned the second-order dynamics on the tangent bundle. Interestingly, while the training loss for the PINN (orange curve) appears lower, the validation metrics reveal that this is a misleading indicator of solution quality, a phenomenon often attributed to the PINN optimizer settling into a non-physical local minimum that satisfies the residual equation pointwise but violates global consistency.

The {middle row} demonstrates the efficacy of the \textit{Cauchy double-integration} reconstruction. The PITDN solution $\tilde{u}_\theta$ perfectly reproduces the standing wave pattern. The error map is extremely low (max amplitude $\approx 0.01$ relative to a signal of $1.0$), and the slices show strictly no phase shift. This proves that the integral operator effectively filters out high-frequency spectral noise that typically plagues second-order derivative estimations. The {bottom row} reveals the limitations of the baseline. While the PINN captures the general standing wave mode, it exhibits visible "wobbling" artifacts in the error map (bottom center), indicating phase errors and amplitude inaccuracies. The quantitative metrics (bottom right) seal the verdict: PITDN achieves a relative $L^2$ error of $\mathbf{7.3 \times 10^{-3}}$, compared to $2.1 \times 10^{-2}$ for the PINN. By achieving an accuracy improvement of nearly $\mathbf{3\times}$ on this challenging second-order problem, PITDN proves that lifting the learning target to the acceleration field is a robust strategy for dispersive wave equations.

\begin{figure}[!ht]
    \centering
    \includegraphics[width=1\linewidth]{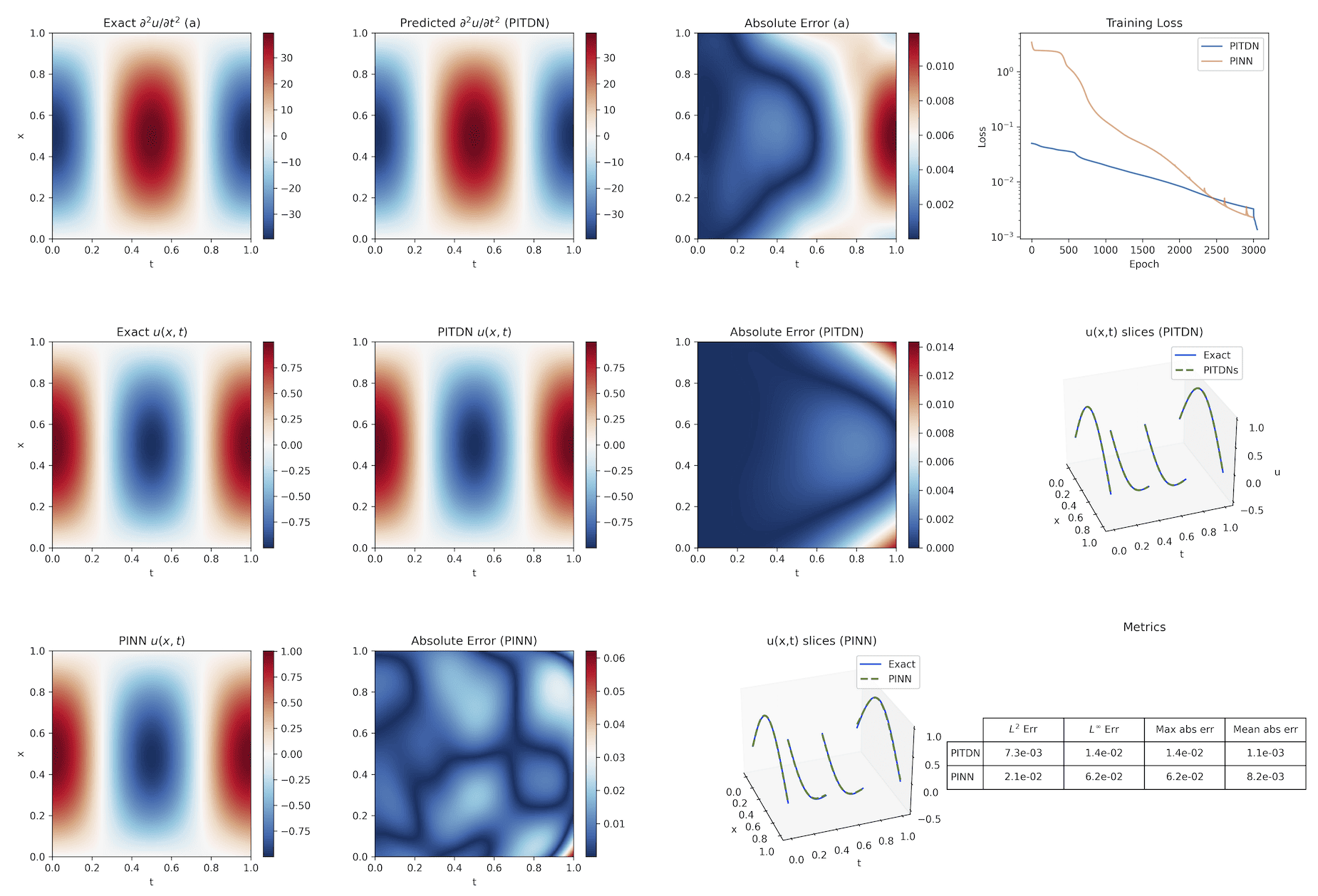}
    \caption{Experimental results for the Nonlinear Klein-Gordon equation using second-order reconstruction. {Top Row:} The PITDN accurately predicts the high-amplitude acceleration field $\partial_{tt}u$. {Middle Row:} Reconstruction via Cauchy's double integration formula yields a precise solution with minimal error ($1.4\%$). {Bottom Row:} The PINN baseline exhibits higher residual errors and phase inaccuracies. Despite a lower training loss, the PINN's generalization error is worse, confirming that PITDN's integral constraint enforces better physical validity.}
    \label{fig:klein_gordon}
\end{figure}

\subsection{Discussion}\label{sec:discussion}

The experimental results provide compelling evidence for the theoretical advantages of the PITDN framework. We now analyze the key factors underlying the observed performance gains, examine limitations, and outline directions for future research.

\subsubsection{Key Advantages}

The PITDN framework naturally aligns with the geometric structure of dynamical systems. By parameterizing the velocity field $v_\theta \approx \partial_t u$ rather than the state directly, the network learns an intrinsically causal representation. The Volterra operator $\mathcal{I}_{u_0}$ serves as a continuous analogue of time-stepping integrators, with dynamics learned adaptively through gradient descent. Crucially, PITDN satisfies $u(x,0) = u_0(x)$ exactly by construction, eliminating the multi-objective optimization problem that plagues standard PINNs.

The differentiated residual $r_t = D_t \mathcal{R}[u]$ amplifies high-frequency error modes, mitigating spectral bias~\cite{rahaman2019}. This is evident in the Burgers' equation experiment, where PITDN captures the steep shock front while the standard PINN produces an overly diffused solution. Despite the Volterra integral operator, computational overhead remains modest at 15--20\% per iteration, offset by faster convergence. The method achieves one to two orders of magnitude improvement in accuracy using only a compact architecture (3 layers, 10 neurons), demonstrating that performance stems from geometric alignment rather than model capacity.

\subsubsection{Limitations and open challenges}

Several limitations warrant discussion. The current formulation requires known initial conditions $u_0(x)$, which limits applicability to inverse problems where $u_0$ must be inferred from sparse observations. Extending the framework to treat $u_0$ as a latent variable remains an open challenge. For problems with discontinuous initial data or singular source terms, weak formulations may be necessary, as the operator $\mathcal{N}[u_0]$ may not be classically well-defined.

Scalability to high dimensions ($d \geq 3$) requires further investigation. While the Volterra operator scales linearly with quadrature points, collocation point density may grow exponentially. Combining PITDN with importance sampling~\cite{nabian2021efficient} or tensor decompositions could address this. The theoretical equivalence in Theorem~\ref{thm:equivalence} assumes exact residual minimization, whereas practice involves approximations. Rigorous a posteriori error estimates would strengthen the foundations and guide adaptive refinement strategies.

\subsubsection{Connections and future directions}

The PITDN framework connects naturally to emerging methodologies in scientific machine learning. Combining PITDN with operator learning frameworks such as DeepONet~\cite{lu2021deeponet} or Fourier Neural Operators~\cite{li2020fourier} could enable rapid inference across parametric families while retaining theoretical guarantees. The tangent bundle perspective also resonates 
with Neural ODEs~\cite{chen2018neural}, though PITDN leverages explicit PDE structure rather than learning latent dynamics from data. For Hamiltonian systems, replacing the trapezoidal rule with symplectic integrators could guarantee conservation properties, benefiting applications in plasma physics and celestial mechanics. Bayesian extensions via variational inference could quantify epistemic uncertainty, crucial for safety-critical applications.

Future work should address coupled systems (Navier-Stokes, Maxwell's equations), adaptive temporal refinement, high-dimensional scaling through dimensionality reduction, inverse problems, stochastic PDEs, and hybrid physics-data models. Ultimately, this work demonstrates that architectural choices grounded in geometric structure can systematically overcome optimization challenges in neural PDE solvers. By aligning representational biases with physical causality, we advance toward machine learning architectures that are accurate, 
interpretable, and fundamentally respectful of mathematical laws.

\section{Conclusion}

Physics-Informed Neural Networks have emerged as a promising mesh-free approach for solving partial differential equations, yet their application to time-dependent problems remains hindered by spectral bias and violation of causality principles. This work addresses these fundamental limitations through a paradigm shift: learning temporal derivatives rather than states directly, thereby lifting the optimization problem to the tangent bundle of the solution manifold.

The proposed Physics-Informed Time Derivative Networks framework offers three key advantages. First, initial conditions are satisfied exactly by construction through a Volterra integral reconstruction operator, eliminating multi-objective optimization conflicts. Second, the time-differentiated residual formulation acts as a natural high-pass filter, providing stronger gradient signals for high-frequency modes typically suppressed during standard training. Third, the approach respects causality by anchoring the solution trajectory at the initial time and learning its rate of evolution.

Experimental validation on hyperbolic, parabolic, and dispersive equations demonstrates accuracy improvements of one to two orders of magnitude compared to standard baselines. These gains are achieved using compact architectures, confirming that performance stems from geometric alignment rather than model capacity. The method proves particularly effective for shock-capturing and long-time integration, where conventional approaches fail.

Future work should address extensions to high-dimensional systems and conservation laws, integration with operator learning frameworks for rapid parametric inference, and incorporation of structure-preserving properties such as symplectic integration for Hamiltonian systems. Ultimately, this work demonstrates that principled architectural choices grounded in geometric structure can systematically overcome optimization challenges in neural PDE solvers. By aligning representational biases with physical causality and temporal evolution, we advance toward machine learning architectures that are accurate, interpretable, stable, and fundamentally respectful of the mathematical laws they encode.

\section*{Declarations}

\subsection*{Conflict of interest}
The authors declare that there is no conflict of interest regarding the publication of this paper.

\subsection*{Funding}
No funding

\subsection*{Data Availability}
The source code and datasets generated during the current study are publicly available in the GitHub repository maintained by the author C.M. at: \url{https://github.com/MCDev30/pitdn}.

\subsection*{Acknowledgments}
The authors would like to thank the anonymous reviewers for their valuable comments and constructive suggestions, which significantly improved the clarity of the methodology and the presentation of the results.

\bibliography{ref}

@article{raissi2019physics,
  title={Physics-informed neural networks: A deep learning framework for solving forward and inverse problems involving nonlinear partial differential equations},
  author={Raissi, Maziar and Perdikaris, Paris and Karniadakis, George E},
  journal={Journal of Computational Physics},
  volume={378},
  pages={686--707},
  year={2019},
  publisher={Elsevier},
  doi={10.1016/j.jcp.2018.10.045},
  url={https://doi.org/10.1016/j.jcp.2018.10.045}
}

@article{lu2021deeponet,
  title={DeepONet: Learning nonlinear operators for identifying differential equations based on the universal approximation theorem of operators},
  author={Lu, Lu and Jin, Pengzhan and Pang, Guofei and Zhang, Zhongqiang and Karniadakis, George Em},
  journal={Nature Machine Intelligence},
  volume={3},
  number={3},
  pages={218--229},
  year={2021},
  publisher={Nature Publishing Group}
}

@article{wang2021failure,
  title={When and why PINNs fail to train: A neural tangent kernel perspective},
  author={Wang, Sifan and Yu, Xinling and Perdikaris, Paris},
  journal={Journal of Computational Physics},
  volume={449},
  pages={110768},
  year={2022},
  publisher={Elsevier}
}

@article{li2020fourier,
  title={Fourier neural operator for parametric partial differential equations},
  author={Li, Zongyi and Kovachki, Nikola and Azizzadenesheli, Kamyar and Liu, Burigede and Bhattacharya, Kaushik and Stuart, Andrew and Anandkumar, Anima},
  journal={arXiv preprint arXiv:2010.08895},
  year={2020}
}

@book{strikwerda2004,
  title={Finite difference schemes and partial differential equations},
  author={Strikwerda, John C.},
  year={2004},
  publisher={SIAM}
}

@book{brenner2008,
  title={The mathematical theory of finite element methods},
  author={Brenner, Susanne and Scott, Ridgway},
  year={2008},
  publisher={Springer Science \& Business Media}
}

@book{boyd2001,
  title={Chebyshev and Fourier spectral methods},
  author={Boyd, John P.},
  year={2001},
  publisher={Courier Corporation}
}

@book{canuto2006,
  title={Spectral methods: fundamentals in single domains},
  author={Canuto, Claudio and Hussaini, M. Yousuff and Quarteroni, Alfio and Zang, Thomas A.},
  year={2006},
  publisher={Springer}
}

@book{bellman1961,
  title={Adaptive control processes: a guided tour},
  author={Bellman, Richard},
  year={1961},
  publisher={Princeton University Press}
}

@article{cybenko1989,
  title={Approximation by superpositions of a sigmoidal function},
  author={Cybenko, George},
  journal={Mathematics of control, signals and systems},
  volume={2},
  number={4},
  pages={303--314},
  year={1989}
}

@article{raissi2017I,
  title={Physics-informed deep learning (part I): Data-driven solutions of nonlinear partial differential equations},
  author={Raissi, Maziar and Perdikaris, Paris and Karniadakis, George E.},
  journal={arXiv preprint arXiv:1711.10561},
  year={2017}
}

@article{raissi2017II,
  title={Physics-informed deep learning (part II): Data-driven discovery of nonlinear partial differential equations},
  author={Raissi, Maziar and Perdikaris, Paris and Karniadakis, George E.},
  journal={arXiv preprint arXiv:1711.10566},
  year={2017}
}

@inproceedings{krishnapriyan2021,
  title={Characterizing possible failure modes in physics-informed neural networks},
  author={Krishnapriyan, Aditi et al.},
  booktitle={Advances in Neural Information Processing Systems},
  volume={34},
  pages={26548--26560},
  year={2021}
}

@inproceedings{rahaman2019,
  title={On the spectral bias of neural networks},
  author={Rahaman, Nasim et al.},
  booktitle={International Conference on Machine Learning},
  pages={5301--5310},
  year={2019}
}

@book{brezis2011,
  title={Functional Analysis, Sobolev Spaces and Partial Differential Equations},
  author={Brezis, Ha{\"i}m},
  year={2011},
  publisher={Springer},
  address={New York, NY},
  series={Universitext},
  isbn={978-0-387-70913-0},
  doi={10.1007/978-0-387-70914-7}
}

@article{devore2021,
  title={Neural network approximation},
  author={DeVore, Ronald and Hanin, Boris and Petrova, Guergana},
  journal={Acta Numerica},
  volume={30},
  pages={327--444},
  year={2021},
  publisher={Cambridge University Press},
  doi={10.1017/S0962492921000052}
}

@book{lang1995,
  title={Differential and Riemannian Manifolds},
  author={Lang, Serge},
  year={1995},
  edition={3rd},
  publisher={Springer},
  address={New York, NY},
  series={Graduate Texts in Mathematics},
  volume={160},
  isbn={978-0-387-94338-1},
  doi={10.1007/978-1-4612-4182-9}
}

@book{dacorogna2007,
  title={Direct Methods in the Calculus of Variations},
  author={Dacorogna, Bernard},
  year={2007},
  edition={2nd},
  publisher={Springer},
  address={New York, NY},
  series={Applied Mathematical Sciences},
  volume={78},
  isbn={978-0-387-35779-9},
  doi={10.1007/978-0-387-55249-1}
}

@book{nocedal2006,
  title={Numerical Optimization},
  author={Nocedal, Jorge and Wright, Stephen J.},
  year={2006},
  edition={2nd},
  publisher={Springer},
  address={New York, NY},
  series={Springer Series in Operations Research and Financial Engineering},
  isbn={978-0-387-30303-1},
  doi={10.1007/978-0-387-40065-5}
}

@article{karimi2016linear,
  title={Linear convergence of gradient and proximal-gradient methods under the Polyak-{\L}ojasiewicz condition},
  author={Karimi, Hamed and Nutini, Julie and Schmidt, Mark},
  journal={Machine Learning and Knowledge Discovery in Databases},
  pages={795--811},
  year={2016},
  publisher={Springer},
  doi={10.1007/978-3-319-46128-1_50}
}

@article{yarotsky2017,
  title={Error bounds for approximations with deep ReLU networks},
  author={Yarotsky, Dmitry},
  journal={Neural Networks},
  volume={94},
  pages={103--114},
  year={2017},
  publisher={Elsevier},
  doi={10.1016/j.neunet.2017.07.002}
}

@book{atkinson1989,
  title={An Introduction to Numerical Analysis},
  author={Atkinson, Kendall E.},
  year={1989},
  edition={2nd},
  publisher={John Wiley \& Sons},
  address={New York},
  isbn={978-0-471-62489-9}
}

@article{lesaffre2001number,
  title={On the number of integration points for Gauss quadrature in nonlinear mixed effect models},
  author={Lesaffre, Emmanuel and Spiessens, Bart},
  journal={Journal of Computational and Graphical Statistics},
  volume={10},
  number={3},
  pages={589--605},
  year={2001},
  publisher={Taylor \& Francis},
  doi={10.1198/106186001317114992}
}

@article{schmidt2020nonparametric,
  title={Nonparametric regression using deep neural networks with ReLU activation function},
  author={Schmidt-Hieber, Johannes},
  journal={The Annals of Statistics},
  volume={48},
  number={4},
  pages={1875--1897},
  year={2020},
  publisher={Institute of Mathematical Statistics},
  doi={10.1214/19-AOS1875}
}

@article{xu2019,
  title={Frequency principle: Fourier analysis sheds light on deep neural networks},
  author={Xu, Zhi-Qin John et al.},
  journal={arXiv preprint arXiv:1901.06523},
  year={2019}
}

@article{daw2023,
  title={Rethinking the importance of sampling in physics-informed neural networks},
  author={Daw, Arka et al.},
  journal={arXiv preprint arXiv:2307.03920},
  year={2023}
}

@article{mcclenny2023,
  title={Self-adaptive physics-informed neural networks},
  author={McClenny, Levi and Braga-Neto, Ulisses},
  journal={Journal of Computational Physics},
  volume={474},
  pages={111722},
  year={2023}
}

@article{xiang2022,
  title={Self-adaptive loss balanced physics-informed neural networks},
  author={Xiang, Zixue et al.},
  journal={Neurocomputing},
  volume={496},
  pages={11--34},
  year={2022}
}

@article{wight2020,
  title={Solving Allen-Cahn and Cahn-Hilliard equations using the adaptive physics informed neural networks},
  author={Wight, Christopher L. and Zhao, Jia},
  journal={arXiv preprint arXiv:2007.04542},
  year={2020}
}

@article{jagtap2020,
  title={Conservative physics-informed neural networks on discrete domains for conservation laws},
  author={Jagtap, Ameya D. and Kharazmi, Ehsan and Karniadakis, George E.},
  journal={Computer Methods in Applied Mechanics and Engineering},
  volume={365},
  pages={113028},
  year={2020}
}

@article{meng2020,
  title={PPINN: Parareal physics-informed neural network for time-dependent PDEs},
  author={Meng, Xuhui et al.},
  journal={Computer Methods in Applied Mechanics and Engineering},
  volume={370},
  pages={113250},
  year={2020}
}

@article{han2018,
  title={Solving high-dimensional partial differential equations using deep learning},
  author={Han, Jiequn and Jentzen, Arnulf and E, Weinan},
  journal={Proceedings of the National Academy of Sciences},
  volume={115},
  number={34},
  pages={8505--8510},
  year={2018}
}

@article{nabian2021efficient,
  title={Efficient training of physics-informed neural networks via importance sampling},
  author={Nabian, Mohammad Amin and Gladstone, Rini Jasmine and Meidani, Hadi},
  journal={Computer-Aided Civil and Infrastructure Engineering},
  year={2021}
}

@inproceedings{chen2018neural,
  title={Neural ordinary differential equations},
  author={Chen, Ricky TQ and Rubanova, Yulia and Bettencourt, Jesse and Duvenaud, David K},
  booktitle={NeurIPS},
  year={2018}
}

@book{ciarlet2002,
  author    = {Philippe G. Ciarlet},
  title     = {The finite element method for elliptic problems},
  publisher = {SIAM},
  year      = {2002}
}

@book{leveque2002,
  author    = {Randall J. LeVeque},
  title     = {Finite volume methods for hyperbolic problems},
  publisher = {Cambridge University Press},
  year      = {2002}
}

@book{gottlieb1977,
  author    = {David Gottlieb and Steven A. Orszag},
  title     = {Numerical analysis of spectral methods: theory and applications},
  publisher = {SIAM},
  year      = {1977}
}

@article{weinan2017,
  author  = {Weinan E},
  title   = {The dawning of a new era in applied mathematics},
  journal = {Notices of the American Mathematical Society},
  year    = {2017}
}

@article{hornik1991,
  author  = {Kurt Hornik},
  title   = {Approximation capabilities of multilayer feedforward networks},
  journal = {Neural Networks},
  volume  = {4},
  number  = {2},
  pages   = {251--257},
  year    = {1991}
}

@article{karniadakis2021,
  author  = {George E. Karniadakis and Ioannis G. Kevrekidis and Lu Lu and Paris Perdikaris and Sifan Wang and Liu Yang},
  title   = {Physics-informed machine learning},
  journal = {Nature Reviews Physics},
  volume  = {3},
  number  = {6},
  pages   = {422--440},
  year    = {2021}
}

@article{kissas2020,
  author  = {Athanasios B. Kissas and Yibo Yang and Eileen Hwuang and Walter R. Witschey and John A. Detre and Paris Perdikaris},
  title   = {Machine learning in cardiovascular flows modeling: Predicting arterial blood pressure from non-invasive 4D flow MRI data using physics-informed neural networks},
  journal = {Computer Methods in Applied Mechanics and Engineering},
  volume  = {358},
  pages   = {112623},
  year    = {2020}
}

@article{shukla2020,
  author  = {Khemraj Shukla and Patricio Clark {Di Leoni} and James Blackshire and Daniel Sparkman and George E. Karniadakis},
  title   = {Physics-informed neural network for ultrasound nondestructive quantification of surface breaking cracks},
  journal = {Journal of Nondestructive Evaluation},
  volume  = {39},
  number  = {3},
  pages   = {1--20},
  year    = {2020}
}

@article{lu2019,
  author  = {Lu Lu and Xuhui Meng and Zhiping Mao and George E. Karniadakis},
  title   = {{DeepXDE}: A deep learning library for solving differential equations},
  journal = {SIAM Review},
  volume  = {63},
  number  = {1},
  pages   = {208--228},
  year    = {2021}
}

@article{wang2021understanding,
  author  = {Sifan Wang and Yujun Teng and Paris Perdikaris},
  title   = {Understanding and mitigating gradient flow pathologies in physics-informed neural networks},
  journal = {SIAM Journal on Scientific Computing},
  volume  = {43},
  number  = {5},
  pages   = {A3055--A3081},
  year    = {2021}
}

@article{wang2022and,
  author  = {Sifan Wang and Hanwen Wang and Paris Perdikaris},
  title   = {On the eigenvector bias of {Fourier} feature networks: From regression to solving physics-informed neural networks},
  journal = {Computer Methods in Applied Mechanics and Engineering},
  volume  = {384},
  pages   = {113940},
  year    = {2021}
}

@article{shukla2021,
  author  = {Khemraj Shukla and Ameya D. Jagtap and George E. Karniadakis},
  title   = {Parallel physics-informed neural networks via domain decomposition},
  journal = {Journal of Computational Physics},
  volume  = {447},
  pages   = {110683},
  year    = {2021}
}

@article{kovachki2021,
  author  = {Nikola Kovachki and others},
  title   = {Neural operator: Graph kernel network for partial differential equations},
  journal = {arXiv preprint arXiv:2003.03485},
  year    = {2021}
}

@inproceedings{kingma2014,
  author    = {Diederik P. Kingma and Jimmy Ba},
  title     = {Adam: A method for stochastic optimization},
  booktitle = {International Conference on Learning Representations (ICLR)},
  year      = {2015}
}

@article{liu1989,
  author  = {Dong C. Liu and Jorge Nocedal},
  title   = {On the limited memory {BFGS} method for large scale optimization},
  journal = {Mathematical Programming},
  volume  = {45},
  number  = {1},
  pages   = {503--528},
  year    = {1989}
}
\bibliographystyle{plain}

\end{document}